\def\mapr{0}

\if\mapr1

\documentclass[smallextended]{svjour3}
\smartqed  
\usepackage{fix-cm}
\usepackage{mathptmx}      

\else

\documentclass{article}
\usepackage[margin=1in]{geometry}
\usepackage{amsthm}

\theoremstyle{plain}
\newtheorem{proposition}{Proposition}
\newtheorem{lemma}{Lemma}
\newtheorem{theorem}{Theorem}

\theoremstyle{definition}
\newtheorem{definition}{Definition}

\theoremstyle{remark}
\newtheorem*{example}{Example}
\newtheorem*{remark}{Remark}

\fi

\usepackage{bm}
\usepackage{amsmath}
\usepackage{amssymb}
\usepackage{amsfonts}

\usepackage[backref=page]{hyperref}
\usepackage[width=12cm,format=plain,labelfont=sf,bf,textfont=sf,up,justification=justified,labelsep=period]{caption}
\usepackage{graphicx}
\usepackage{framed}
\usepackage{xcolor}
\usepackage[mathscr]{euscript} 

\hypersetup{
    colorlinks=true,       
    linkcolor=blue,          
    citecolor=blue,        
    filecolor=blue,      
    urlcolor=blue           
}

\newcommand{\cX}{\mathcal X}
\newcommand{\cY}{\mathcal Y}

\newcommand{\cA}{\mathcal A}

\newcommand{\RR}{\mathbb{R}}

\renewcommand{\S}{{\bf S}} 

\DeclareMathOperator{\rank}{rank}
\DeclareMathOperator{\cprank}{cprank}

\DeclareMathOperator{\supp}{supp}

\DeclareMathOperator{\conv}{conv}
\DeclareMathOperator{\diag}{diag}

\DeclareMathOperator{\fool}{fool} 
\DeclareMathOperator{\1}{\mathbf{1}}

\let\vec\relax 
\DeclareMathOperator{\vec}{vec}

\DeclareMathOperator{\aff}{aff} 

\newcommand{\mb}[1]{\bm{#1}} 

\newcommand{\cp}{\text{\upshape cp}} 
\newcommand{\sos}{\text{\upshape sos}} 
\newcommand{\fra}{\text{\upshape frac}} 
\newcommand{\bartheta}{\bar{\vartheta}}
\DeclareMathOperator{\RG}{RG} 
\newcommand{\pn}{\mathscr{N}} 
\newcommand{\pns}{\mathscr{N}^{\ast}} 
\newcommand{\swapix}[3]{#1[#3 : #2]} 
\newcommand{\ecc}{\text{\upshape cl}} 

\title{
Self-scaled bounds for atomic cone ranks:\\
applications to nonnegative rank and cp-rank
}
\if\mapr1
\author{Hamza Fawzi \and Pablo A. Parrilo}
\institute{Laboratory for
  Information and Decision Systems (LIDS), Massachusetts Institute of
  Technology, Cambridge, MA 02139, USA.\\
\email{\{hfawzi,parrilo\}@mit.edu}}
\else
\author{Hamza Fawzi \and Pablo A. Parrilo\thanks{The authors are with
    the Laboratory for Information and Decision Systems, Department of
    Electrical Engineering and Computer Science, Massachusetts
    Institute of Technology, Cambridge, MA 02139. Email:
    \texttt{\{hfawzi,parrilo\}@mit.edu}.}}
\renewcommand\footnotemark{}
\fi
\date{April 11th, 2014}

\begin{document}
\maketitle

\begin{abstract}
The nonnegative rank of a matrix $A$ is the smallest integer $r$ such that $A$ can be written as the sum of $r$ rank-one nonnegative matrices. The nonnegative rank has received a lot of attention recently due to its application in optimization, probability and communication complexity. In this paper we study a class of atomic rank functions defined on a convex cone which generalize several notions of ``positive'' ranks such as nonnegative rank or cp-rank (for completely positive matrices). The main contribution of the paper is a new method to obtain lower bounds for such ranks which improve on previously known bounds. Additionally the bounds we propose can be computed by semidefinite programming using sum-of-squares relaxations. The idea of the lower bound relies on an atomic norm approach where the atoms are \emph{self-scaled} according to the vector (or matrix, in the case of nonnegative rank) of interest. This results in a lower bound that is invariant under scaling and that is at least as good as other existing norm-based bounds.

We mainly focus our attention on the two important cases of nonnegative rank and cp-rank where our bounds have an appealing connection with existing combinatorial bounds and satisfy some additional interesting properties. For the nonnegative rank we show that our lower bound can be interpreted as a non-combinatorial version of the fractional rectangle cover number, while the sum-of-squares relaxation is closely related to the Lov\'asz $\bartheta$ number of the rectangle graph of the matrix. The self-scaled property also implies that the lower bound is at least as good as other norm-based bounds on the nonnegative rank.  Finally we prove that the lower bound inherits many of the structural properties satisfied by the nonnegative rank such as invariance under diagonal scaling, subadditivity, etc.
We also apply our method to obtain lower bounds on the cp-rank for completely positive matrices. In this case we prove that our lower bound is always greater than or equal the plain rank lower bound, and we show that it has interesting connections with combinatorial lower bounds based on edge-clique cover number.
\end{abstract}


\section{Introduction}

\paragraph{Preliminaries} Given an elementwise nonnegative matrix $A \in \RR^{m\times n}_+$, a \emph{nonnegative factorization} of $A$ of size $r$ is a decomposition of $A$ of the form:
\[ A = \sum_{i=1}^r u_i v_i^T, \]
where $u_i\in \RR^m, v_i \in \RR^n$ are elementwise nonnegative vectors. The \emph{nonnegative rank} of $A$, denoted $\rank_+(A)$ is the smallest size of a nonnegative factorization of $A$. Observe that the following inequalities always hold:
\[ \rank(A) \leq \rank_+(A) \leq \min(n,m). \]
The nonnegative rank plays an important role in statistical modeling \cite{drton2009lectures,kubjas2013fixed}, communication complexity \cite{lovasz1990communication,lee2009lower} and optimization \cite{yannakakis1991expressing,gouveia2011lifts}. 
In probability and statistics, a nonnegative matrix $A\in \RR^{\cX \times \cY}_+$ has a natural interpretation as the joint distribution of a pair of random variables $(X,Y)$, i.e.,
\[ A(x,y) = \Pr[X=x,Y=y], \]
for all $x \in \cX, y \in \cY$ (the matrix $A$ is assumed to be normalized so that its elements all sum to one). Under this interpretation, the constraint $\rank_+(A) \leq r$ encodes the fact that the pair $(X,Y)$ is a \emph{mixture of $r$ independent random variables} on $\cX\times \cY$. Indeed, since the elements of $A$ sum up to one, any nonnegative factorization of $A$ can be normalized appropriately so that it takes the form:
\begin{equation}
 \label{eq:mixtureP}
 A = \sum_{i=1}^r \lambda_i u_i v_i^T,
\end{equation}
where the coefficients $\lambda_i \geq 0$ sum up to one, and where $u_i \in \RR^{\cX}_+$ and $v_i \in \RR^{\cY}_+$ are nonnegative vectors with $\1^T u_i = \1^T v_i = 1$. 
Each rank-one term $u_i v_i^T$ corresponds to a distribution on $\cX\times \cY$ which is independent, and thus Equation \eqref{eq:mixtureP} expresses the fact that $A$ is the mixture of $r$ independent distributions on $\cX\times \cY$.

The nonnegative rank has a natural generalization to \emph{tensors}. Given a nonnegative tensor $A = [a_{i_1\dots i_n}]$ of size $d_1\times \dots \times d_n$, the nonnegative rank of $A$ is the smallest $r$ for which there exists a decomposition of $A$ of the form:
\[ A = \sum_{i=1}^r u_{1,i} \otimes u_{2,i} \otimes \dots \otimes u_{n,i}, \]
where for each $i=1,\dots,r$ the vectors $u_{1,i} \in \RR^{d_1},\dots,u_{n,i}\in\RR^{d_n}$ are elementwise nonnegative. When all the entries of $A$ sum up to one, $A$ can be seen as the joint probability distribution of $n$ random variables $(X_1,\dots,X_n)$. The set of nonnegative tensors with nonnegative rank less than or equal $r$ corresponds precisely to the joint distributions that are mixtures of $r$ independent distributions (cf. \cite{drton2009lectures}).

\paragraph{General framework} In this paper we present a new method to obtain lower bounds on the nonnegative rank. 
In fact, the method we introduce applies in general to any \emph{atomic rank function} associated with a convex cone. We make the \emph{atomic rank} notion precise in the following definition:
\begin{definition}
Let $K$ be a convex cone and $V$ be a given algebraic variety in some Euclidean space. Given $A \in K$ we define $\rank_{K,V}(A)$ to be the smallest integer $r$ for which we can write
\[ A = \sum_{i=1}^r R_i \]
where each $R_i \in K\cap V$. The function $\rank_{K,V}$ is called the \emph{atomic rank function} associated to $K$ and $V$.
\end{definition}
Different well-known notions of rank fit into this general framework:
\begin{itemize}
\item {\bf Sparsity of a nonnegative vector}: Let $K$ be the nonnegative orthant in $\RR^n$, i.e., $K = \RR^n_+$, and let $V$ be the variety of vectors having at most one nonzero component, i.e.,
\[ V = (\RR e_1) \cup \dots \cup (\RR e_n) = \{ x \in \RR^n \; : \; x_i x_j = 0 \;\; \forall 1\leq i < j \leq n \}, \]
where $e_1,\dots,e_n$ are the vectors of the canonical basis. Then for this choice of $K$ and $V$ the rank of an element $x \in K$ is the sparsity of the vector $x$, i.e., the number of nonzero components of $x$.
\item {\bf Nonnegative rank}: Let $K$ be the cone of nonnegative matrices in $\RR^{m\times n}$, i.e., $K = \RR^{m\times n}_+$ and let $V$ be the variety of rank-one matrices:
\[ V = \{ R \in \RR^{m\times n} \; : \; \rank R = 1 \}. \]
Then one can verify that $\rank_{K,V}(A)$ for $A \in K$ is precisely the nonnegative rank of $A$.

\item {\bf The plain rank of a symmetric positive semidefinite matrix}: When $A$ is a real symmetric positive-semidefinite $n\times n$ matrix, an important fact in linear algebra states that the (standard) rank of $A$ can be defined as the smallest $r$ such that we have:
\[ A = \sum_{i=1}^r R_i, \]
where the $R_i$'s are rank-one and \emph{symmetric positive-semidefinite} (what is remarkable here is that the rank-one terms $R_i$ can be taken to be symmetric and positive semidefinite). Thus if we choose $K = \S^n_+$ (the cone of real symmetric positive-semidefinite matrices) and $V$ to be the variety of rank-one matrices, then $\rank_{K,V}(A)$ is nothing but the (standard) rank of the matrix $A \in \S^n_+$.
\item {\bf CP-rank for completely-positive matrices}: A symmetric matrix $A \in \S^n$ is called \emph{completely-positive} \cite{berman2003completely} if it admits a decomposition of the form:
\[ A = \sum_{i=1}^r u_i u_i^T, \]
where the vectors $u_i$ are nonnegative. The \emph{cp-rank} of $A$ is defined as the smallest $r$ for which such a decomposition of $A$ exists. It corresponds to the atomic rank where $K$ is the cone of completely positive matrices, and $V$ is the variety of rank-one matrices.

\item {\bf Sums of even powers of linear forms}: Let $\RR[\mb{x}]_{2d}$ be the space of homogeneous polynomials of degree $2d$ in $n$ variables $\mb{x}=(x_1,\dots,x_n)$ . Let $L_{n,2d}$ be the cone of homogeneous polynomials that can be written as the sum of $2d$'th powers of linear forms, i.e., $P \in L_{n,2d}$ if:
\begin{equation}
\label{eq:decompP}
 P(\mb{x}) = \sum_{i=1}^r \ell_i(\mb{x})^{2d},
\end{equation}
where $\ell_i(\mb{x})$ are linear forms. In \cite{reznick1992sums}, Reznick studied a quantity which he denoted by $w(P)$ and is defined as the smallest number of terms in any decomposition of $P$ of the form \eqref{eq:decompP}. It is easy to see that $w(P)$ is exactly $\rank_{K,V}(P)$ where $K=L_{n,2d}$ and $V$ is the variety of $2d$'th powers of linear forms\footnote{It is known that the space $\RR[\mb{x}]_{2d}$ can be identified with the space of symmetric tensors of size $n\times n\times \dots \times n$ ($2d$ dimensions), see e.g., \cite[Section 3.1]{comon2008symmetric}. Then one can verify that a polynomial $P$ is the $2d$'th power of a linear form if and only if the tensor associated to $P$ is of the form $\ell \otimes \ell \otimes \dots \otimes \ell$.} (also known as the Veronese variety).
\begin{remark}
The quantity $w(P)$ is related to the real \emph{Waring rank} of homogeneous polynomials, see e.g., \cite{landsberg2012tensors,blekherman2014maximum}: the real Waring rank of a homogeneous polynomial $P$ of degree $k$ is the size of the smallest decomposition of $P$ as a linear combination of $k$'th powers of linear forms. The case considered above corresponds to the situation where $k=2d$ is even, and where the coefficients in the linear combination are required to be nonnegative.
\end{remark}
\begin{remark}
Let $P_{n,2d}$ be the cone of nonnegative polynomials in $\RR[\mb{x}]_{2d}$. It is known, see e.g., \cite[Section 4.4.2]{frgbook} that the cone $L_{n,2d}$ can be identified, via the apolar inner product in $\RR[\mb{x}]_{2d}$, with $P_{n,2d}^*$ the dual of the cone of nonnegative polynomials. The extreme rays of $P_{n,2d}^*$ correspond to point evaluations. Thus, using this dual point of view, the atomic rank of an element $\ell \in P_{n,2d}^*$ is the smallest $r$ such that $\ell$ can be written as a conic combination of $r$ point evaluations. For example if $\ell$ is an integral operator $\ell:P\mapsto \int P(x) d\mu(x)$ where $\mu$ is a positive measure on the unit sphere $\mathbb{S}^{n-1}$, then $\rank_{K,V}(\ell)$ gives the size of the smallest cubature formula of order $2d$ \cite{konig1999cubature} for the measure $\mu$.
\end{remark}
\end{itemize}

\paragraph{Self-scaled bounds} We now briefly explain the main idea of the lower bound in the general framework considered above: Let $A \in K$ and consider a decomposition of $A$ of the form:
\begin{equation}
 \label{eq:decompPK}
 A = \sum_{i=1}^r R_i, \qquad R_i \in V \cap K \; \; \forall i=1,\dots,r,
\end{equation}
An important observation is that each term $R_i$ in the decomposition above necessarily satisfies 
\[ 0\preceq_{K} R_i \preceq_{K} A \]
where $\preceq_{K}$ denotes the inequality induced by the cone $K$ (recall that $x \preceq_{K} y \Leftrightarrow y-x \in K$).
Thus, if we define:
\begin{equation}
\label{eq:AKV}
 \cA_{K,V}(A) := \Bigl\{ R \in V \; \text{ such that } \; 0 \preceq_{K} R \preceq_{K} A \Bigr\},
\end{equation}
then in any decomposition of $A$ of the form \eqref{eq:decompPK}, all the terms $R_i$ must necessarily belong to $\cA_{K,V}(A)$. As a consequence, if we can produce a linear functional $L$ such that $L(R) \leq 1$ for all $R \in \cA_{K,V}(A)$, then clearly $L(A)$ is a lower bound on the minimal number of terms in any decomposition of $A$ of the form \eqref{eq:decompPK}. Indeed this is because we have:
\[ L(A) = \sum_{i=1}^r L(R_i) \leq \sum_{i=1}^r 1 = r. \]
In other words, the quantity $L(A)$ gives a lower bound on $\rank_{K,V}(A)$. Now to obtain the best lower bound, one can look for the linear functional $L$ which maximizes the value of $L(A)$ while satisfying $L \leq 1$ on $\cA_{K,V}(A)$. We call this quantity $\tau_{K,V}(A)$ and this is the main object we study in this paper:
\begin{equation}
\label{eq:tauKV}
 \tau_{K,V}(A) := \max_{L \text{ linear}} \;\; L(A) \quad \text{ subject to } \quad L(R) \leq 1 \;\;\; \forall R \in \cA_{K,V}(A).
\end{equation}
The discussion above shows that $\tau_{K,V}(A)$ gives a lower bound on $\rank_{K,V}(A)$.
\begin{theorem}
Let $K$ be a convex cone and $V$ a given algebraic variety. Then for any $A \in K$ we have
\[ \rank_{K,V}(A) \geq \tau_{K,V}(A). \]
\end{theorem}
The idea of the lower bound described above may look similar to existing lower-bounding techniques based on dual norms like e.g., in \cite{lee2009lower} or \cite{doan2013finding}. The main difference however is the \emph{self-scaled}\footnote{We use the word \emph{self-scaled} as a descriptive term to convey the main idea of the lower bound presented in this paper. It is not related to the term as used in the context of interior-point methods (e.g., ``self-scaled barrier'' \cite{nesterov1997self}).} property of our lower bound: in other words, the specific normalization of the set of atoms $\cA_{K,V}(A)$ \emph{depends} on the element $A$, whereas in the other techniques the atoms are normalized with respect to some fixed norm (e.g., the $\ell_2$ norm, the $\ell_{\infty}$ norm, etc.), and \emph{independently} of $A$. In fact for this reason one can show that our lower bound is at least as good as any other lower bound obtained using norm-based methods (cf. Section \ref{sec:normbased} for more details).



\paragraph{Semicontinuity of atomic cone ranks} We saw that in any decomposition of the form \eqref{eq:decompPK}, each term $R_i$ must satisfy $0 \preceq_K R_i \preceq_K A$ and is thus \emph{bounded} (assuming $K$ is a pointed cone). Using this observation, one can show that atomic rank functions are lower semi-continuous, or equivalently, that the sets $\{ A \in K \; : \; \rank_{K,V}(A) \leq r\}$ are closed for any $r \geq 1$.
 This property was noted before in \cite{bocci2011perturbation,lim2009nonnegative} in the particular case of the nonnegative rank. Note that the positivity condition on the $R_i$'s here is crucial. It is well-known for example that the standard tensor rank is not lower semi-continuous for tensors of order $\geq 3$, which leads in this situation to the distinction between the rank and the \emph{border rank} \cite{landsberg2012tensors}.

\paragraph{Nonnegative rank} We now briefly discuss the specialization of our lower bound to the case of nonnegative rank.
As we mentioned, the case of nonnegative rank of matrices corresponds to the choices $K=\RR^{m\times n}_+$ (nonnegative matrices) and $V$ is the variety of rank-one matrices.
In this case we denote the set of atoms $\cA_{K,V}(A)$ simply by $\cA_+(A)$ and the quantity $\tau_{K,V}$ simply by $\tau_+$:
\begin{equation}
\label{eq:A+intro}
\cA_+(A) := \Bigl\{ R \in \RR^{m\times n} \; : \; \rank R \leq 1 \; \text{ and } \; 0\leq R \leq A \Bigr\},
\end{equation}
and
\begin{equation}
\label{eq:tau+intro}
 \tau_+(A) := \max_{L \text{ linear}} \;\; L(A) \quad \text{ subject to } \quad L(R) \leq 1 \;\;\; \forall R \in \cA_{+}(A).
\end{equation}
As defined above, the quantity $\tau_+(A)$ cannot be efficiently computed since we do not have an efficient description of the feasible set $\{  L \text{ linear} \; : \; L(R) \leq 1 \; \forall R \in \cA_+(A) \}$, even though \eqref{eq:tau+intro} is a convex optimization problem. We thus propose a semidefinite programming relaxation, denoted $\tau_+^{\sos}(A)$ which is obtained by relaxing the constraint $L\leq 1$ on $\cA_+(A)$ using sum-of-squares methods (the exact definition of this relaxation is presented in more details in Section \ref{sec:tau+sdprelaxation}). We study various properties of the quantities $\tau_+(A)$ and $\tau_+^{\sos}(A)$ and we show for example that they are invariant under diagonal scaling and that they satisfy many of the structural properties satisfied by the nonnegative rank (subadditivity, etc.), cf. Theorem \ref{thm:tau+properties}.

We then compare $\tau_+$ and $\tau_+^{\sos}$ with existing bounds on the nonnegative rank and we show that they have very interesting connections to well-known combinatorial bounds. Indeed we show that $\tau_+(A)$ can be understood as a non-combinatorial version of the \emph{fractional chromatic number} of the \emph{rectangle graph} of $A$ (also called the \emph{fractional rectangle cover number} of $A$), while $\tau_+^{\sos}(A)$ is the non-combinatorial equivalent of the (complement) \emph{Lov\'{a}sz theta number} of the rectangle graph of $A$ \cite{fiorini2013combinatorial}. In fact we show that:
\[ \tau_+(A) \geq \chi_{\fra}(\RG(A)) \quad \text{ and } \quad \tau_+^{\sos}(A) \geq \bartheta(RG(A)), \]
where $\RG(A)$ denotes the rectangle graph associated to $A$ and $\chi_{\fra}$ and $\bartheta$ denote, respectively, the \emph{fractional chromatic number} and the (complement) Lov\'{a}sz theta number (more details concerning the definition of rectangle graph and the various graph parameters are in Section \ref{sec:tau+comparison-combinatorial-lbs}).

Finally we compare our new lower bounds with other norm-based (non-combinatorial) bounds on $\rank_+(A)$ such as the ones proposed in \cite{fawzi2012new} or \cite{braun2012approximation} (see also Lemma 4 in \cite{rothvoss2013matching}). Using the ``self-scaled'' property of our bound, we prove a general result showing that $\tau_+$ always yields better bounds that any such norm-based method.


\paragraph{Organization} The paper is organized as follows: In Section~\ref{sec:rank+matrices} we consider the nonnegative rank of matrices where we study the quantity $\tau_+$ as well as its semidefinite programming relaxation $\tau_+^{\sos}$.  We prove various properties on these two quantities and we compare them with existing combinatorial and norm-based bounds on the nonnegative rank. We conclude the section with some numerical examples illustrating the performance of the lower bound. In Section~\ref{sec:rank+tensors} we discuss the generalization of the nonnegative rank lower bound to tensors and we evaluate it numerically on an example. Finally, in Section~\ref{sec:cprank} we deal with the problem of cp-rank for completely positive matrices: we present the definition of the lower bound as well as its semidefinite programming relaxation and we explore some of its interesting properties. We show the surprising fact that the lower bound is always at least as good as the plain rank lower bound and we also discuss connections with combinatorial lower bounds. We conclude the section with some numerical experiments.

We provide Matlab scripts for the numerical examples shown in this paper at the URL \url{http://www.mit.edu/~hfawzi}. The scripts make use of the Yalmip package \cite{YALMIP} for solving the semidefinite programs.

\paragraph{Notations}
We denote by $\S^n$ the space of real symmetric $n\times n$ matrices, and by $\S^n_+$ the cone of real symmetric positive semidefinite matrices. If $A \in \RR^{m\times n}$ is a $m\times n$ matrix we define $a=\vec(A) \in \RR^{mn}$ to be the vector of length $mn$ obtained by stacking all the columns of $A$ on top of each other. Recall that if $A$ and $B$ are matrices of size $m\times n$ and $m'\times n'$ respectively, then the Kronecker product $A \otimes B$ is a matrix of size $mm'\times nn'$ matrix defined as follows:
\[
A \otimes B = \begin{bmatrix} A_{1,1} B & A_{1,2} B & \dots & A_{1,n} B\\
A_{2,1} B & A_{2,2} B & \dots & A_{2,n} B\\
\vdots & \vdots & & \vdots\\
A_{m,1} B & A_{n,2} B & \dots & A_{m,n} B \end{bmatrix} \in \RR^{mm'\times nn'}.
\]
When $A,X,B$ are matrices of appropriate sizes, we have the following identity:
\[ \vec(AXB) = (B^T \otimes A) \vec(X). \]
We define the following partial order on the indices of a matrix $A \in \RR^{m\times n}$:
\[ (i,j) \leq (k,l) \; \Leftrightarrow \; i\leq k \text{ and } j \leq l, \]
and we write $(i,j) < (k,l)$ if $i < k$ and $j < l$.
If $n$ is an integer, we let $[n]:=\{1,\dots,n\}$.

We use the notation $(\RR^d)^*$ to denote the dual space of $\RR^d$ which consists of linear functionals on $\RR^d$. We recall some terminology from convex analysis \cite{rockafellar1997convex}:  If $C$ is a convex set in $\RR^d$, we denote by $C^{\circ} \subset (\RR^d)^*$ the \emph{polar} of $C$ defined by: $C^{\circ} = \{ \ell \in (\RR^d)^* \; : \; \ell(x) \leq 1 \; \forall x \in C\}$. The \emph{support function} $S_C:(\RR^d)^* \rightarrow \RR$ of a convex set $C$ is defined as $S_C(\ell) = \max_{x \in C} \ell(x)$. The Minkowski gauge function of $C$ is defined as $p_C(x) = \min\{t > 0 \; : \; x \in tC\}$.



\section{Nonnegative rank of matrices}
\label{sec:rank+matrices}


\subsection{Primal and dual formulations for $\tau_+$}

For a nonnegative matrix $A \in \RR^{m\times n}_+$ , recall the following definitions from the introduction:
\begin{definition}
Given a nonnegative matrix $A \in \RR^{m\times n}_+$, we define $\cA_+(A)$ to be the set of rank-one nonnegative matrices $R$ that satisfy $0\leq R \leq A$:
\[ \cA_+(A) := \Bigl\{ R \in \RR^{m\times n} \; : \; \rank R \leq 1 \; \text{ and } \; 0\leq R \leq A \Bigr\}. \]
We also let
\begin{equation}
\label{eq:tau+max}
 \tau_+(A) := \max_{L \text{ linear}} \;\; L(A) \quad \text{ subject to } \quad L(R) \leq 1 \;\;\; \forall R \in \cA_{+}(A).
\end{equation}
\end{definition}
\begin{theorem}
For $A \in \RR^{m\times n}_+$ we have $\rank_+(A) \geq \tau_+(A)$.
\end{theorem}
\begin{proof}
Let $A = \sum_{i=1}^r R_i$ be a nonnegative factorization of $A$ with $r=\rank_+(A)$ and $R_i \geq 0$ are rank-one. Then necessarily each $R_i$ satisfies $R_i\leq A$ and thus $R_i \in \cA_+(A)$ for all $i$. Hence if $L$ is the optimal solution in the definition of $\tau_+(A)$ we get:
\[ \tau_+(A) = L(A) = \sum_{i=1}^r L(R_i) \leq r = \rank_+(A). \]
\if\mapr1\qed\fi
\end{proof}

\paragraph{Minimization formulation of $\tau_+$} Using convex duality, one can obtain a dual formulation of $\tau_+(A)$ as the solution of a certain minimization problem. In fact the next lemma shows that $\tau_+(A)$ is nothing but the \emph{atomic norm} \cite{chandrasekaran2012convex} associated to the set of atoms $\cA_+(A)$. This interpretation of $\tau_+(A)$ will be very useful later when studying its properties.
\begin{lemma}
\label{lem:duality}
If $A \in \RR^{m\times n}_+$ then we have:
\begin{equation}
 \label{eq:tau+min}
 \tau_+(A) = \min \{ t > 0 \; : \; A \in t \conv(\cA_+(A)) \}.
\end{equation}
In other words, $\tau_+(A)$ is the Minkowski gauge function of $\conv(\cA_+(A))$, evaluated at $A$.
\end{lemma}
\begin{proof}
Observe that Equation \eqref{eq:tau+max} expresses the fact that $\tau_+(A)$ is the support function of $\conv(\cA_+(A))^{\circ}$, evaluated at $A$. Theorem 14.5 in \cite{rockafellar1997convex} shows that the support function of the polar $C^{\circ}$ of a closed convex set $C$ is equal to the Minkowski gauge function of $C$. Thus it follows that $\tau_+(A)$ is equal to the Minkowski gauge function of $\conv(\cA_+(A))$, evaluated at $A$, which is precisely Equation \eqref{eq:tau+min}.
\if\mapr1\qed\fi
\end{proof}

The next example illustrates the geometric picture underlying the atomic norm formulation of $\tau_+(A)$.
\begin{example}
Assume $A$ is a $2\times 2$ diagonal matrix $A = \diag(a_1,a_2)$ where $a_i \geq 0$. In this case one can easily verify that $\cA_+(A)$ is given by:
\begin{equation}
\label{eq:A+diagonal}
\begin{aligned}
 \cA_+(A) &= \left\{ R \in \RR^{2\times 2} \; : \; \rank R \leq 1 \text{ and } 0 \leq R \leq \begin{bmatrix} a_1 & 0\\ 0 & a_2 \end{bmatrix}  \right\}\\
&= \left\{ \begin{bmatrix} x & 0\\ 0 & 0 \end{bmatrix} \text{ with } 0\leq x \leq a_1 \right\} \cup \left\{ \begin{bmatrix} 0 & 0\\ 0 & y \end{bmatrix} \text{ with } 0\leq y \leq a_2 \right\}.
\end{aligned}
\end{equation}

The convex hull of $\cA_+(A)$ (projected onto the diagonal elements) is depicted in Figure \ref{fig:example_atoms}. 
\begin{figure}[ht]
  \centering
  \includegraphics[width=7cm]{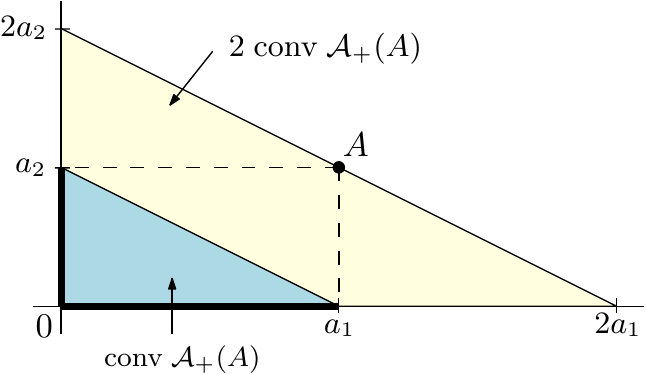}
  \caption{Depiction of the set of atoms $\cA_+(A)$ and its convex hull for a $2\times 2$ diagonal matrix $A$ (cf. Equation \eqref{eq:A+diagonal}). The set $\cA_+(A)$ consists of the two dark heavy lines joining the origin to $a_1$ and $a_2$. The convex hull of $\cA_+(A)$ is formed by the triangle $0,a_1,a_2$.}
  \label{fig:example_atoms}
\end{figure}
Observe that, when $a_1,a_2 > 0$, the smallest $t$ such that $A \in t \conv(\cA_+(A))$ is $t=2$ and thus $\tau_+(A) = 2=\rank_+(A)$. In fact we will see in the next section that when $A$ is a diagonal matrix, $\tau_+(A)$ is precisely equal to the number of nonzero elements on the diagonal, which is equal to $\rank_+(A)$.
\end{example}

We can see in this example the \emph{self-scaled} feature of the bound $\tau_+(A)$. This is in contrast with the existing norm-based methods to lower bound $\rank_+(A)$ such as \cite{fawzi2012new,braun2012approximation}, where the scaling of the atoms is \emph{independent} of $A$: for example in \cite{fawzi2012new} the scaling is done using the Frobenius norm (i.e., the set of atoms consists of rank-one matrices with Frobenius norm) and in \cite{braun2012approximation} the scaling is with respect to the entry-wise infinity norm. This feature is explained in more detail in Section \ref{sec:normbased} where we show that $\tau_+$ always yields better bounds than any such norm-based method.

\subsection{Semidefinite programming relaxation}
\label{sec:tau+sdprelaxation}

The quantity $\tau_+(A)$ defined in \eqref{eq:tau+max} cannot be efficiently computed in general, since we do not have an efficient description of the feasible $\{ L \text{ linear} \; : \; L(R) \leq 1 \; \forall R \in \cA_+(A) \}$ (note however that \eqref{eq:tau+max} is a convex optimization problem). In this section we introduce a semidefinite programming relaxation of $\tau_+(A)$. To do so we construct an over-relaxation of the set $\conv \cA_+(A)$ which can be represented using linear matrix inequalities. Recall that $\cA_+(A)$ is the intersection of the variety of rank-one matrices with the set $\{ R \in \RR^{m\times n} \; : \; 0\leq R \leq A \}$. The variety of rank-one matrices is described by the vanishing of $2\times 2$ minors, i.e.,
\begin{equation}
 \label{eq:2x2minors}
 R_{i,j} R_{k,l} - R_{i,l} R_{k,j} = 0
\end{equation}
for all $(1,1)\leq (i,j) < (k,l) \leq (m,n)$ (recall the partial order $(i,j) < (k,l) \Leftrightarrow i < k \text{ and } j < l$). Let $r =\vec(R)$ be the vector obtained by stacking all columns of $R$ and consider the following positive-semidefinite matrix:
\begin{equation}
\label{eq:matrix-rrT}
\begin{bmatrix} 1\\ r \end{bmatrix} \begin{bmatrix} 1 \\ r \end{bmatrix}^T = \begin{bmatrix} 1 & r^T \\ r & rr^T \end{bmatrix}.
\end{equation}
Note that $rr^T$ is a symmetric $mn \times mn$ matrix whose rows and columns are indexed by entries of $R$. The quadratic equations \eqref{eq:2x2minors} corresponding to the vanishing of $2\times 2$ minors of $R$ can be written as linear equations in the entries of $rr^T$, namely:
\[ (rr^T)_{ij,kl} - (rr^T)_{il,kj} = 0 \]
for $(1,1)\leq (i,j) < (k,l) \leq (m,n)$ (in the equation above, the subscripts ``$ij$'' and ``$kl$'' in $(rr^T)_{ij,kl}$ are the indices in $\{1,\dots,mn\}$ for the entries $(i,j)$ and $(k,l)$ respectively---we will use this slight abuse of notation in the paper to avoid having heavy notations). \\
Also note that the inequality $R \leq A$ implies that:
\begin{equation}
 \label{eq:diag-ineqs}
 (rr^T)_{ij,ij} \leq r_{ij} A_{ij}
\end{equation}
which is a linear inequality in the entries of the matrix \eqref{eq:matrix-rrT}. Using these two observations we have the following over-relaxation of $\conv(\cA_+(A))$:
\begin{equation}
 \label{eq:overrelaxationA}
 \conv(\cA_+(A)) \subseteq \cA_+^{\sos}(A)
\end{equation}
where
\begin{equation}
\label{eq:convAsos}
\begin{aligned}
\cA_+^{\sos}(A) = \Biggl\{ R \in \RR^{m\times n} \; : \; & \exists X \in \S^{mn} \; \text{ such that } \; \begin{bmatrix}
1 & \vec(R)^T\\
\vec(R) & X
\end{bmatrix} \succeq 0\\
& \qquad \text{ and } X_{ij,ij} \leq R_{ij} A_{ij} \quad \forall i \in [m],j \in [n]\\
& \qquad \text{ and } X_{ij,kl} - X_{il,kj} = 0 \quad \forall (1,1)\leq (i,j) < (k,l) \leq (m,n)\Biggr\}.
\end{aligned}
\end{equation}
If we define $\tau_+^{\sos}(A)$ as:
\[ \tau_+^{\sos}(A) = \min \{ t > 0 \; : \; A \in t\cA^{\sos}_+(A) \} \]
then we clearly have (by the inclusion \eqref{eq:overrelaxationA}):
\[ \tau_+^{\sos}(A) \leq \tau_+(A) \leq \rank_+(A). \]
Furthermore, the quantity $\tau_+^{\sos}(A)$ can be computed using semidefinite programming. Indeed, it is not difficult to show using the description \eqref{eq:convAsos} of $\cA_+^{\sos}(A)$ that we have:
\begin{equation}
\label{eq:tau+sos_min}
\begingroup
\renewcommand*{\arraystretch}{1.3}
\begin{array}{rrl}
\tau_+^{\sos}(A) &= 
\;\;\underset{t,X}{\text{min}} & t\\
&\text{s.t.} &
\begin{bmatrix} t & \vec(A)^T\\ \vec(A) & X \end{bmatrix} \succeq 0\\
&& X_{ij,ij} \leq A_{ij}^2 \quad \forall i\in[m],j\in[n]\\
&& X_{ij,kl} = X_{il,kj} \quad \forall (1,1)\leq (i,j) < (k,l) \leq (m,n)
\end{array}
\endgroup
\end{equation}
\paragraph{Duality and sum-of-squares interpretation}
The dual of the semidefinite program \eqref{eq:tau+sos_min} takes the form of a \emph{sum-of-squares} program, namely we have\footnote{The sum-of-squares program \eqref{eq:tau+sos_max} is actually the dual of a slightly different, but equivalent, formulation of \eqref{eq:tau+sos_min} where the inequality $X_{ij,ij} \leq A_{ij}^2$ is replaced by $X_{ij,ij} \leq A_{ij} Y_{ij}$ where $Y_{ij}$ are additional variables that are constrained by $Y_{ij} = A_{ij}$. Using this reformulation, the dual has a nice interpretation as a sum-of-squares relaxation of \eqref{eq:tau+max}. Also one can easily show that strong duality holds using Slater's condition.}:
\begin{equation}
\label{eq:tau+sos_max}
\tau_+^{\sos}(A) = 
\begin{array}[t]{ll}
\text{max} & L(A)\\
\text{s.t.} & L \text{ is a linear form}\\
& 1 - L(X) = SOS(X) + \sum_{ij} D_{ij} X_{ij} (A_{ij} - X_{ij})\;\; \text{ mod } I\\
& D_{ij} \geq 0\\
& SOS(X) \text{ is a sum-of-squares polynomial}
\end{array}
\end{equation}
Here $I$ is the ideal in $\RR[X_{11},\dots,X_{mn}]$ corresponding to the variety of $m\times n$ rank-one matrices, i.e., it is ideal generated by the $2\times 2$ minors $X_{ij} X_{kl} - X_{il} X_{kj}$. The sum-of-squares constraint in \eqref{eq:tau+sos_max}  means that the polynomials on each side of the equality are equal when $X$ is rank-one. Note that this sum-of-squares constraint can be rewritten more explicitly as requiring that:
\[ 1 - L(X) - \sum_{ij} D_{ij} X_{ij} (A_{ij} - X_{ij}) - \sum_{(i,j) < (k,l)} \nu_{ijkl} (X_{ij} X_{kl} - X_{il} X_{kj}) \text{ is a sum-of-squares} \]
where the parameters $\nu_{ijkl}$ are real numbers\footnote{One can show that the sum-of-squares polynomial cannot have degree more than 2 and the multipliers $\nu_{ijkl}$ are necessarily real numbers.}. It is clear that any such $L$ satisfies $L(X) \leq 1$ for all $X \in \cA_+(A)$. As such, \eqref{eq:tau+sos_max} is a natural sum-of-squares relaxation of \eqref{eq:tau+max}.




\paragraph{Zero entries in $A$} When the matrix $A$ has some entries equal to 0, the semidefinite program \eqref{eq:tau+sos_min} that defines $\tau_+^{\sos}(A)$ can be reduced by eliminating unnecessary variables. Let $S=\supp(A)=\{(i,j) \; : \; A_{i,j} > 0\}$ be the set of nonzero entries of $A$, and define $\pi:\RR^{m\times n} \rightarrow \RR^S$ to be the linear map that projects onto the entries in $S$.
Observe that, in the SDP \eqref{eq:tau+sos_min}, if $A_{i,j} = 0$ for some $(i,j)$ then necessarily $X_{ij,ij} = 0$. Thus by the positivity constraint this implies that the $ij$'th row and $ij$'th column of $X$ are identically zero, and one can thus eliminate this row and column from the program. Using this fact, one can show that $\tau_+^{\sos}(A)$ can be computed using the following reduced semidefinite program where the size of the matrix $X$ is now $|\supp(A)|\times |\supp(A)|$, instead of $mn\times mn$ (recall that $\pi(A)$ is the vectorization of $A$ where we only keep the nonzero entries of $A$):
\begin{equation}
\label{eq:tau+sos_min-reduced}
\begin{array}{rrl}
\tau_+^{\sos}(A) &= 
\;\;\underset{t,X}{\text{min}} & t\\
&\text{s.t.} &
\begin{bmatrix} t & \pi(A)^T\\ \pi(A) & X \end{bmatrix} \succeq 0\\
&& \forall (i,j) \text{ s.t. } A_{i,j} > 0: \; X_{ij,ij} \leq A_{ij}^2\\
&& \forall (1,1)\leq (i,j) < (k,l) \leq (m,n) \text{ s.t. } A_{i,j} A_{k,l} > 0 \text{ or } A_{i,l} A_{k,j} > 0:\\
&& \qquad
\begin{cases}
 \text{if } A_{i,l} A_{k,j}=0\; : \; X_{ij,kl} = 0\\
 \text{if } A_{i,j} A_{k,l}=0 \; : \; X_{il,kj} = 0\\
 \text{else } \;\; X_{ij,kl} - X_{il,kj} = 0
\end{cases}
\end{array}
\end{equation}

\subsection{Properties}

\label{sec:tau+properties}

In this section we explore some of the properties of $\tau_+(A)$ and $\tau_+^{\sos}(A)$. We show that $\tau_+(A)$ and $\tau_+^{\sos}(A)$ have many appealing properties (invariance under diagonal scaling, invariance under  permutation, etc.) which are not present in most of currently existing bounds on the nonnegative rank. The theorem below summarizes the desirable properties satisfied by $\tau_+(A)$ and $\tau_+^{\sos}(A)$.

\begin{theorem}
\label{thm:tau+properties}
Let $A \in \RR^{m\times n}_+$ be a nonnegative matrix.
\begin{enumerate}
\item Invariance under diagonal scaling: If $D_1$ and $D_2$ are diagonal matrices with strictly positive entries on the diagonal, then $\tau_+(D_1 A D_2) = \tau_+(A)$ and $\tau_+^{\sos}(D_1 A D_2) = \tau_{+}^{\sos}(A)$.
\item Invariance under permutation of rows or columns: If $P_1$ and $P_2$ are permutation matrices of size $m\times m$ and $n \times n$ respectively, then $\tau_+(P_1 A P_2) = \tau_+(A)$ and $\tau_+^{\sos}(P_1 A P_2) = \tau_+^{\sos}(A)$.
\item Subadditivity: If $B \in \RR^{m\times n}_+$ is a nonnegative matrix then:
\[ \tau_+(A+B) \leq \tau_+(A)+\tau_+(B) \quad \text{ and } \quad \tau_+^{\sos}(A+B) \leq \tau_+^{\sos}(A) + \tau_+^{\sos}(B). \]
\item Product: If $B \in \RR^{n \times p}_+$, then
\[ \tau_+(AB) \leq \min(\tau_+(A),\tau_+(B)) \quad \text{ and } \quad \tau_+^{\sos}(AB) \leq \min(\tau_+^{\sos}(A),\tau_+^{\sos}(B)). \]
\item Monotonicity: If $B$ is a submatrix of $A$ (i.e., $B = A[I,J]$ for some $I\subseteq [m]$ and $J\subseteq [n]$), then $\tau_+(B) \leq \tau_+(A)$ and $\tau_+^{\sos}(B) \leq \tau_+^{\sos}(A)$.
\item Block-diagonal composition: Let $B \in \RR^{m'\times n'}_+$ be another nonnegative matrix and define
\[ A \oplus B = \begin{bmatrix} A & 0\\ 0 & B \end{bmatrix}. \]
Then
\[ \tau_+(A\oplus B) = \tau_+(A) + \tau_+(B) \quad \text{ and } \quad
   \tau_+^{\sos}(A\oplus B) = \tau_+^{\sos}(A) + \tau_+^{\sos}(B) \]
\end{enumerate}
\end{theorem}

Before proving the theorem, we look at some of the existing bounds on the nonnegative rank in light of the properties listed in the theorem above.
\begin{itemize}
\item {\bf Norm-based lower bounds:} In a previous paper \cite{fawzi2012new} we introduced a lower bound on $\rank_+(A)$ based on the idea of a \emph{nonnegative nuclear norm}. We showed that:
\begin{equation}
 \label{eq:lbnu+}
 \rank_+(A) \geq \left(\frac{\nu_+(A)}{\|A\|_F}\right)^2
\end{equation}
where $\nu_+(A)$ is the \emph{nonnegative nuclear norm} of $A$ defined by:
\[ \nu_+(A) = \max_{W \in \RR^{m \times n}} \; \left\{ \; \langle A, W \rangle \; : \; \begin{bmatrix} I & -W\\ -W^T & I \end{bmatrix} \text{ copositive} \; \right\} \]
and where $\|A\|_F = \sqrt{\sum_{i,j} A_{i,j}^2}$ is the Frobenius norm.
We showed with an example (cf. Example 5 in \cite{fawzi2012new}) that the lower bound can change when applying diagonal scaling to the matrix $A$; in other words the bound \eqref{eq:lbnu+} is not invariant under diagonal scaling.\\
Also it is known that nuclear-norm based lower bounds are not monotone in general, i.e., the value of the bound can be greater when applied to a submatrix of $A$ (cf. \cite[Section 2.3.2]{lee2009lower}).
\item {\bf Lower bounds from information theory:} Information theoretic quantities can be used to get lower bounds on the nonnegative rank; in fact such bounds were used recently in  \cite{braverman2013information,braun2013common} in the context of extended formulations of polytopes. For example the Shannon mutual information as well as Wyner's common information \cite{wyner1975common} provide lower bounds on the nonnegative rank. However as we show below these bounds are not invariant under diagonal scaling. We first recall the definition of these lower bounds. Let $P \in \RR^{\cX \times \cY}_+$ be a nonnegative matrix such that $\sum_{x,y} P(x,y) = 1$ and let $(X,Y)$ be a pair of random variables distributed according to $P$:
\[ \Pr[X=x,Y=y] = P(x,y)  \quad \forall x \in \cX, y \in \cY. \]
Recall from the introduction that a nonnegative factorization of $P$ of size $k$ expresses the fact that $(X,Y)$ is a mixture of $k$ independent random variables on $\cX \times \cY$, i.e., we can write:
\[ \Pr[X=x,Y=y] = \sum_{w=1}^k \Pr[W=w] \cdot \Pr[X=x|W=w] \cdot \Pr[Y=y|W=w], \]
where $W$ is the mixing distribution, taking values in $\{1,\dots,k\}$ and $X$ and $Y$ are conditionally independent given $W$.
 Using this interpretation, the nonnegative rank of $P$ can thus be formulated as:
\begin{equation}
 \label{eq:rank+itformulation}
 \rank_+(P) = \min_{\substack{X-W-Y\\ (X,Y) \sim P}} |\supp(W)|,
\end{equation}
where $|\supp(W)|$ is the number of values that $W$ takes, and where the Markov chain constraint $X-W-Y$ means that $X$ and $Y$ are conditionally independent given $W$.\\
Using the formulation \eqref{eq:rank+itformulation} one can easily obtain information-theoretic lower bounds on $\rank_+(P)$. In fact one can show using simple information-theoretic inequalities that
\[ I(X;Y) \leq C(X;Y) \leq \log \rank_+(P) \]
where $I(X;Y)$ is the Shannon mutual information and $C(X;Y)$ is Wyner's common information \cite{wyner1975common} defined by:
\[ C(X;Y) = \min_{X-W-Y} I(XY;W). \]
The lower bounds $I(X;Y)$ and $C(X;Y)$ however are not invariant under diagonal scaling (here the scaling is followed by a global normalization of the matrix to make the sum of its entries equal to one): indeed if $P$ is a diagonal matrix, $P = \diag(p)$ where $p > 0$ and $1^T p = 1$ then one can show that
\begin{equation} \label{eq:ICdiag} I(X;Y) = C(X;Y) = H(p) \end{equation}
where $H$ denotes Shannon entropy, $H(p) = -\sum_{i} p_i \log p_i$. Now note that any quantity defined on nonnegative matrices and which is invariant under diagonal scaling should take the same value on diagonal matrices that have strictly positive entries on the diagonal (this is because if $P$ and $Q$ are diagonal matrices then we can transform $P$ into $Q$ by a diagonal scaling). Equation \eqref{eq:ICdiag} however shows that the quantities $I(X;Y)$ and $C(X;Y)$ depend on the specific values on the diagonal and thus are  not invariant under diagonal scaling.
\end{itemize}

We now turn to the proof of Theorem \ref{thm:tau+properties}. We prove below the first property (invariance under diagonal scaling) and we prove the remaining properties in Appendix \ref{sec:appendix-props}.

\begin{proof}[Proof of invariance under diagonal scaling]
\begin{enumerate}
\item We first prove the property for $\tau_+$. Let $A' = D_1 A D_2$ where $D_1$ and $D_2$ are the two diagonal matrices with strictly positive entries on the diagonal. Observe that the set of atoms $\cA_+(A')$ of $A'$ can be obtained from the atoms $\cA_+(A)$ of $A$ as follows:
\begin{equation}
\label{eq:atomsA'-diagonalscaling}
\cA_+(A') = \{ D_1 R D_2 \; : \; R \in \cA_+(A) \} =: D_1 \cA_+(A) D_2.
\end{equation}
Indeed, if $R$ is rank-one and $0 \leq R \leq A$ then clearly $D_1 R D_2$ is rank-one and satisfies $0 \leq D_1 R D_2 \leq D_1 A D_2 = A'$ thus $D_1 R D_2 \in \cA_+(A')$. Conversely if $R' \in \cA_+(A')$, then by letting $R = D_1^{-1} R D_2^{-1}$ we see that $R' = D_1 R D_2$ with $R$ rank-one and $0\leq R \leq A$. Thus this shows equality \eqref{eq:atomsA'-diagonalscaling}.
Thus we have:
\begin{equation}
\label{eq:chain-equalities-invariance-diagona-scaling}
\begin{aligned}
\tau_+(A') &= \min \left\{ t \; : \; A' \in t \cdot \conv(\cA_+(A')) \right\}\\
           &= \min \left\{ t \; : \; D_1 A D_2 \in t \cdot \conv(D_1 \cA_+(A) D_2) \right\}\\
           &= \min \left\{ t \; : \; D_1 A D_2 \in t \cdot D_1 \conv(\cA_+(A)) D_2 \right\}\\
           &= \min \left\{ t \; : \; A \in t \cdot  \conv(\cA_+(A)) \right\}\\
           &= \tau_+(A).
\end{aligned}
\end{equation}
\item We now prove the property for the SDP relaxation $\tau_+^{\sos}$. For this we use the maximization formulation of $\tau_+^{\sos}$ given in Equation \eqref{eq:tau+sos_max}. Let $L$ be the optimal linear form in \eqref{eq:tau+sos_max} for the matrix $A$, i.e., $L(A) = \tau_+^{\sos}(A)$ and $L$ satisfies:
\begin{equation}
 \label{eq:certL}
 1 - L(X) = SOS(X) + \sum_{ij} D_{ij} X_{ij} (A_{ij} - X_{ij}) \bmod I.
\end{equation}
Define the linear polynomial $L'(X) = L(D_1^{-1} X D_2^{-1})$. It is straightforward to see from \eqref{eq:certL} that $L'$ satisfies:
\[ 1 - L'(X) = SOS(D_1^{-1} X D_2^{-1}) + \sum_{ij} \frac{D_{ij}}{(D_1)_{ii}^2 (D_2)_{jj}^2} X_{ij} ( A'_{ij} - X_{ij} ) \bmod I. \]
Thus this shows that $L'$ is feasible for the sum-of-squares program \eqref{eq:tau+sos_max} for the matrix $A'$. Thus since $L'(A') = L(A) = \tau_+^{\sos}(A)$, we get that $\tau_+^{\sos}(A') \geq \tau_+^{\sos}(A)$.  With the same reasoning we can show that:
\[ \tau_+^{\sos}(A) = \tau_+^{\sos}(D_1^{-1}(D_1 A D_2) D_2^{-1}) \geq \tau_+^{\sos}(D_1 A D_2) = \tau_+^{\sos}(A'). \]
Thus we have $\tau_+^{\sos}(A') = \tau_+^{\sos}(A)$.

\end{enumerate}
\if\mapr1\qed\fi
\end{proof}



%
%

\subsection{Discussion on the SDP relaxation}

\paragraph{Additional constraints} 
The semidefinite program \eqref{eq:tau+sos_min} that defines $\tau_+^{\sos}(A)$ can potentially be strengthened by including additional constraints on the matrix $X$. In this paragraph we discuss how these might affect the value of the lower bound.
\begin{itemize}
\item First observe that the constraint \eqref{eq:diag-ineqs} is a special case $k=i,l=j$ of the constraints
\[ (rr^T)_{ij,kl} \leq r_{ij} A_{kl}, \]
for any $i,j,k,l$. This would correspond in the semidefinite program \eqref{eq:tau+sos_min} to adding the inequalities
\[ X_{ij,kl} \leq A_{ij} A_{kl}. \]
In Lemma \ref{lem:offdiag-ineqs} of Appendix \ref{app:product} we show that these inequalities are automatically verified by any $X$ feasible for the semidefinite program \eqref{eq:tau+sos_min}. Thus adding these inequalities does not affect the value of the bound.

\item Another constraint that one can impose on $X$ is \emph{elementwise nonnegativity}, since $rr^T$ in \eqref{eq:matrix-rrT} is nonnegative. We investigated the effect of this constraint numerically but on all the examples we tried the value of the bound did not change (up to numerical precision). We think however there might be specific examples where the value of the bound does change. Indeed as we show later in Section \ref{sec:tau+comparison-combinatorial-lbs} the quantity $\tau_+^{\sos}$ is closely related to the Lov{\'a}sz $\vartheta$ number. It is known in the case of $\vartheta$ that adding a nonnegativity constraint can affect the value of the SDP, even though the change is often not very significant. The version of the $\vartheta$ number with an additional nonnegativity constraint is sometimes denoted by $\vartheta^+$ and was first introduced by Szegedy in \cite{szegedy1994note} and extensive numerical experiments were done in \cite{dukanovic2007semidefinite} (see also \cite{meurdesoif2005strengthening}).
\item Another family of constraints that one could impose in the SDP comes from the following observation: If $0 \leq r \leq a$ (with $r=\vec(R)$ and $a=\vec(A)$) then we have for any $i,j,k,l$:
\[ (a-r)_{ij} (a-r)_{kl} \geq 0, \]
i.e.,
\[ a_{ij} a_{kl} - r_{ij} a_{kl} - r_{kl} a_{ij} + (rr^T)_{ij,kl} \geq 0. \]
In the semidefinite program \eqref{eq:tau+sos_min} these inequalities translate to:
\begin{equation}
 \label{eq:constraint2-t}
 X_{ij,kl} \geq (2-t) A_{ij} A_{kl}.
\end{equation}
We observed that on most matrices the constraint does not affect the value of the bound. However for some specific matrices of small size the value can change:
For example for the matrix
\[ A = \begin{bmatrix} 1 & 1\\ 1 & 1/2 \end{bmatrix} \]
we get the value $4/3$ without the constraint \eqref{eq:constraint2-t} whereas with the constraint we obtain $3/2$. 
\end{itemize}

Despite the possible improvements, the constraints mentioned here would make the size of the semidefinite programs much larger and we have observed that on most examples the value of the lower bound does not change much.
We have also noted that by including some of the additional constraints we lose some of the nice properties that the quantity $\tau_+^{\sos}(A)$ satisfies. For example if we include the last set of inequalities \eqref{eq:constraint2-t} described above, the lower bound is no longer additive for block-diagonal matrices.

\paragraph{Parametrization of the rank-one variety} We saw that the semidefinite programming relaxation of $\tau_+$ corresponds to relaxing the constraint $L(X) \leq 1 \; \forall X \in \cA_+(A)$ by the following sum-of-squares constraint:
\begin{equation}
\label{eq:1-Lcert}
 1-L(X) = SOS(X) + \sum_{ij} D_{ij} X_{ij} (A_{ij} - X_{ij}) \quad \forall X \in \RR^{m\times n} \text{ rank-one}
\end{equation}
where $SOS(X)$ is a sum-of-squares polynomial and $D_{ij}$ are nonnegative real numbers. One way to specify that the equality above has to hold for all $X$ rank-one is to require that the two polynomials on each side of the equality are equal \emph{modulo the ideal} $I$ of rank-one matrices. This is the approach we adopted when presenting the sum-of-squares program for $\tau_+^{\sos}(A)$ in \eqref{eq:tau+sos_max}.

%

Another approach to encode the constraint \eqref{eq:1-Lcert} is to \emph{parametrize} the variety of rank-one matrices: indeed we know that rank-one matrices $X$ have the form $X_{ij} = u_i v_j$ for all $i=1,\dots,m$ and $j=1,\dots,n$ where $u_i, v_j \in \RR$. Thus one way to guarantee that \eqref{eq:1-Lcert} holds is to ask that the following polynomial identity (in the ring $\RR[u_1,\dots,u_m,v_1,\dots,v_n]$) holds:
\begin{equation}
 \label{eq:1-Lcert_param}
 1 - \sum_{ij} L_{ij} u_i v_j = SOS(u,v) + \sum_{ij} D_{ij} u_i v_j (A_{ij} - u_i v_j).
\end{equation}
It can be shown that these two approaches (working modulo the ideal vs. using the parametrization of rank-one matrices) are actually identical and lead to the same semidefinite programs.


\subsection{Comparison with combinatorial lower bounds on nonnegative rank}
\label{sec:tau+comparison-combinatorial-lbs}
Many of the known bounds on the nonnegative rank are combinatorial and depend only on the sparsity pattern of the matrix $A$. These bounds are usually expressed as parameters of some graph constructed from $A$. In this section we explore the connection between the quantities $\tau_+(A)$ and $\tau_+^{\sos}(A)$ and these combinatorial quantities.

Let $A \in \RR^{m\times n}_+$ be a nonnegative matrix. A \emph{monochromatic rectangle} for $A$ is a rectangle $R=I\times J$ such that $A_{i,j} > 0$ for any $(i,j) \in R$, i.e., the rectangle does not touch any zero entry of $A$. Note that in any nonnegative factorization $A = \sum_{i=1}^r u_i v_i^T$, the rectangles $R_i = \supp(u_i) \times \supp(v_i)$ are necessarily monochromatic for $A$. The boolean rank of $A$ (also called the rectangle covering number), denoted $\rank_B(A)$ is the minimum number of monochromatic rectangles needed to cover the nonzero entries of $A$. From the previous observation it is easy to see that $\rank_B(A) \leq \rank_+(A)$.

As noted in \cite{fiorini2013combinatorial} the boolean rank of $A$ can be expressed as the chromatic number of a certain graph constructed from $A$. Define the \emph{rectangle graph} of $A$, denoted $\RG(A)$ as follows: The vertex set of $\RG(A)$ is the set of indices $(i,j)$ such that $A_{i,j} > 0$; furthermore there is an edge (undirected) between vertices $(i,j)$ and $(k,l)$ if, and only if, $A_{i,l} A_{k,j} = 0$. Figure \ref{fig:example_rectangle_graph} below shows an example of a rectangle graph for a $3\times 3$ nonnegative matrix.

\begin{figure}[ht]
  \centering
  \includegraphics[width=8cm]{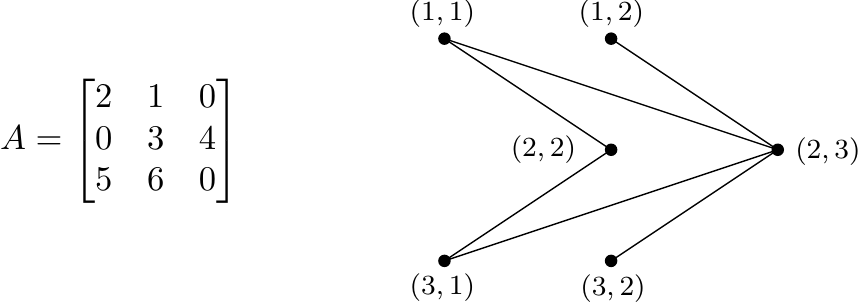}
  \caption{Rectangle graph $\RG(A)$ of a matrix $A$. Note that the graph does not depend on the specific values of $A_{i,j}$, only on the zero/nonzero pattern.}
  \label{fig:example_rectangle_graph}
\end{figure}

Note that if two entries $(i,j)$ and $(k,l)$ of $A$ are connected by an edge in $\RG(A)$, then the two entries cannot be covered by the same monochromatic rectangle.
 Using this observation, it is not hard to show that the minimum number of monochromatic rectangles needed to cover the nonzero entries $A$ is precisely the chromatic number of $\RG(A)$ \cite[Lemma 5.3]{fiorini2013combinatorial}:
\[ \rank_B(A) = \chi(\RG(A)). \]
An obvious lower bound on the chromatic number of $\RG(A)$ is the \emph{clique number} of $\RG(A)$, i.e., the size of the largest clique, which is denoted by $\omega(\RG(A))$. The clique number $\omega(\RG(A))$ is also sometimes known as the \emph{fooling set number} of $A$. Other famous lower bounds on $\chi(\RG(A))$ are the \emph{fractional chromatic number} $\chi_{\fra}(\RG(A))$ and the \emph{(complement) Lov\'{a}sz theta number} $\bartheta(\RG(A))$. These quantities satisfy the following inequalities:
\[ \fool(A) = \omega(\RG(A)) \leq \bartheta(\RG(A)) \leq \chi_{\fra}(\RG(A)) \leq \chi(\RG(A)) = \rank_B(A). \]

We will now see that the quantities $\tau_+(A)$ and $\tau_+^{\sos}(A)$ can be interpreted as non-combinatorial equivalents of $\chi_{\fra}(\RG(A))$ and $\bartheta(\RG(A))$ respectively. We start by recalling the definitions of the fractional chromatic number and the Lov\'{a}sz theta number.

\begin{itemize}
\item The fractional chromatic number of a graph $G$ is a linear programming relaxation of  the chromatic number (note however that the size of this LP relaxation may have exponential size and the fractional chromatic number is actually NP-hard \cite{lund1994hardness}). When applied to the rectangle graph of $A$, the quantity is called the \emph{fractional rectangle cover} of $A$ (see e.g., \cite{karchmer1992fractional}). Let $\cA_{B}(A)$ be the set of monochromatic rectangles valid for $A$ (the subscript ``B'' here stands for ``Boolean''):
\[ 
\cA_{B}(A) = \Bigl\{ R \in \{0,1\}^{m\times n} \; : \; \text{$R$ is a monochromatic rectangle for $A$} \Bigr\}.
\]
Using this notation, the fractional rectangle cover number of $A$ is the solution of the following linear program:
\begin{equation}
\label{eq:fracrc0}
\begingroup
\renewcommand*\arraystretch{1.5}
\chi_{\fra}(\RG(A)) = 
\begin{array}[t]{ll}
\text{min} & \displaystyle\sum_{R \in \cA_{B}(A)} x_R\\
\text{s.t.} & \forall R \in \cA_B(A) \; : \; x_R \geq 0\\
            & \forall (i,j), \;\; A_{i,j} > 0 \; \Rightarrow \; \displaystyle\sum_{R \in \cA_{B}(A)} x_R R_{i,j} \geq 1.
\end{array}
\endgroup
\end{equation}
Note that if we replace the constraint $x_R \geq 0$ with the binary constraint $x_R \in \{0,1\}$, we get the exact rectangle cover number of $A$. We can rewrite the linear program above in the following form, which emphasizes the connection with the quantity $\tau_+(A)$ (cf. Equation \eqref{eq:tau+min}):
\begin{equation}
\label{eq:fracrc}
 \chi_{\fra}(\RG(A)) = 
\begin{array}[t]{ll}
\text{min} & t\\
\text{s.t.} & \exists Y \in t \conv(\cA_B(A)) \; \text{ s.t. } \; \forall (i,j), \;\; A_{i,j} > 0 \; \Rightarrow \; Y_{i,j} \geq 1.
\end{array}
\end{equation}
The variable $Y$ above plays the role of $\sum_{R \in \cA_B(A)} x_R R$ in \eqref{eq:fracrc0}.

Note that a result of Lov\'{a}sz \cite{lovasz1975ratio} shows that for any graph $G=(V,E)$ the fractional chromatic number of $G$ is always within a $\ln |V|$ factor from $\chi(G)$, namely:
\[ \frac{1}{1+\ln|V|} \chi(G) \leq \chi_{\fra}(G) \leq \chi(G).   \]
\item Given a graph $G=(V,E)$, the complement Lov\'{a}sz theta number $\bartheta(G) \overset{\text{def}}{=} \vartheta(\bar{G})$ is defined by the following semidefinite program:
\[
\bartheta(G) = 
\begin{array}[t]{ll}
\text{min} & t\\
\text{subject to} & \begin{bmatrix} t & 1^T\\ 1 & X \end{bmatrix} \succeq 0\\
                  & X_{u,u} = 1 \quad \forall u \in V\\
                  & X_{u,v} = 0 \quad \forall \{u,v\} \in E
\end{array}
\]
When applied to the rectangle graph $\RG(A)$ of a nonnegative matrix $A$, we get:
\begin{equation}
\label{eq:bartheta-RGA}
\bartheta(\RG(A)) = 
\begin{array}[t]{ll}
\text{min} & t\\
\text{subject to} & \begin{bmatrix} t & 1^T\\ 1 & X \end{bmatrix} \succeq 0\\
& \forall (i,j) \text{ s.t. } A_{i,j} > 0: \; X_{ij,ij} = 1\\
& \forall (1,1)\leq (i,j) < (k,l) \leq (m,n):\\
& \qquad
\begin{cases}
 \text{if } A_{i,l} A_{k,j}=0\; : \; X_{ij,kl} = 0 \quad \text{(\ref{eq:bartheta-RGA}a)}\\
 \text{if } A_{i,j} A_{k,l}=0 \; : \; X_{il,kj} = 0 \quad \text{(\ref{eq:bartheta-RGA}b)}
\end{cases}
\end{array}
\end{equation}
Note how the semidefinite program above resembles the semidefinite program \eqref{eq:tau+sos_min-reduced} which defines $\tau_+^{\sos}(A)$. In Theorem \ref{thm:relation-tau-RGA} below, we show in fact that $\tau_+^{\sos}(A) \geq \bartheta(\RG(A))$.
\end{itemize}

\begin{theorem}
\label{thm:relation-tau-RGA}
If $A \in \RR^{m\times n}_+$ is a nonnegative matrix, then
\[ \tau_+(A) \geq \chi_{\fra}(\RG(A)) \quad \text{ and } \quad \tau_+^{\sos}(A) \geq \bartheta(\RG(A)). \]
\end{theorem}
\begin{proof}
\begin{enumerate}
\item 
We prove first that $\tau_+(A) \geq \chi_{\fra}(\RG(A))$. For convenience, we recall below the definitions of $\tau_+(A)$ and $\chi_{\fra}(\RG(A))$:
\[
\begingroup
\renewcommand*\arraystretch{1.5}
\begin{array}{c|c}
\tau_+(A) & \chi_{\fra}(\RG(A))\\ \hline
\begin{array}[t]{ll}
\min & t\\
\text{s.t.} & A \in t \conv(\cA_+(A))
\end{array}
&
\begin{array}[t]{ll}
\text{min} & t\\
\text{s.t.} & \exists Y \in t \conv(\cA_B(A))\\
                  & \quad \text{ s.t. } \forall (i,j), \;\; A_{i,j} > 0 \; \Rightarrow \; Y_{i,j} \geq 1
\end{array}
\end{array}
\endgroup
\]
Let $t = \tau_+(A)$ and $X \in \conv(\cA_+(A))$ such that $A = t X$. Consider the decomposition of $X$:
\[ X = \sum_{k=1}^r \lambda_k X_k, \]
where $X_k \in \cA_+(A)$, $\lambda_k \geq 0$ and $\sum_{k=1}^r \lambda_k = 1$. Let $R_k = \supp(X_k)$ (i.e., $R_k$ is obtained by replacing the nonzero entries of $X_k$ with ones) and observe that $R_k \in \cA_B(A)$. Define
\[ Y = t \sum_{k=1}^r \lambda_k R_k \in t \conv(\cA_B(A)) \]
Observe that for any $(i,j)$ such that $A_{i,j} > 0$ we have:
\[ Y_{i,j} = t \sum_{k:X_k[i,j]>0} \lambda_k \underbrace{R_k[i,j]}_{=1} \overset{(a)}{\geq} t \sum_{k:X_k[i,j]} \lambda_k \frac{X_k[i,j]}{A_{i,j}} \overset{(b)}{=} \frac{A_{i,j}}{A_{i,j}} = 1 \]
where in (a) we used the fact that $X_k \leq A$ (by definition of $X_k \in \cA_+(A)$) and in (b) we used the fact that $A = t\sum_{k} \lambda_k X_k$. Thus this shows that $(t,Y)$ is feasible for the optimization program defining $\chi_{\fra}(\RG(A))$ and thus we have $\chi_{\fra}(\RG(A)) \leq t = \tau_+(A)$.

\item We now show that $\tau_+^{\sos}(A) \geq \bartheta(\RG(A))$. For convenience, we recall the two SDPs \eqref{eq:tau+sos_min-reduced} and \eqref{eq:bartheta-RGA} that define $\tau_+^{\sos}(A)$ and $\bartheta(\RG(A))$ below (note the constraint $X_{ij,ij} = A_{ij}^2$ in the SDP on the left appears as an inequality constraint in \eqref{eq:tau+sos_min-reduced}---in fact it is not hard to see that with an equality constraint we get the same optimal value):
\[
\begingroup
\renewcommand*\arraystretch{1.5}
\begin{array}{c|c}
\tau_+^{\sos}(A) & \bartheta(\RG(A))\\ \hline
\begin{array}[t]{rl}
\text{min.} & t\\
\text{s.t.} & \begin{bmatrix} t & \pi(A)^T\\ \pi(A) & X \end{bmatrix} \succeq 0 \\
& \forall (i,j) \text{ s.t. } A_{i,j} > 0: \;\; X_{ij,ij} = A_{ij}^2\\
& \forall (1,1)\leq (i,j) < (k,l) \leq (m,n):\\
& \qquad
\begin{cases}
 \text{if } A_{i,l} A_{k,j}=0\; : \; X_{ij,kl} = 0 \quad \text{(a)}\\
 \text{if } A_{i,j} A_{k,l}=0 \; : \; X_{il,kj} = 0 \quad \text{(b)}\\
 \text{else } \;\; X_{ij,kl} - X_{il,kj} = 0 \quad \text{(c)}
\end{cases}
\end{array}
&
\begin{array}[t]{rl}
\text{min.} & t\\
\text{s.t.} 
& \begin{bmatrix} t & 1^T\\ 1 & X \end{bmatrix} \succeq 0 \\
& \forall (i,j) \text{ s.t. } A_{i,j} > 0: \;\; X_{ij,ij} = 1\\
& \forall (1,1)\leq (i,j) < (k,l) \leq (m,n):\\
& \qquad
\begin{cases}
 \text{if } A_{i,l} A_{k,j}=0\; : \; X_{ij,kl} = 0 \quad \text{(a')}\\
 \text{if } A_{i,j} A_{k,l}=0 \; : \; X_{il,kj} = 0 \quad \text{(b')}\\
\end{cases}
\end{array}

\end{array}
\endgroup
\]
 Observe that the two semidefinite programs are very similar except that $\tau_+^{\sos}(A)$ has more constraints than $\bartheta(\RG(A))$; cf. constraints (c) for $\tau_+^{\sos}(A)$. To show that $\tau_+^{\sos}(A) \geq \bartheta(\RG(A))$, let $(t,X)$ be the solution of the SDP on the left for $\tau_+^{\sos}(A)$. We will construct $X'$ such that $(t,X')$ is feasible for the SDP on the right and thus this will show that $\tau_+^{\sos}(A) \geq \bartheta(\RG(A))$.
Define $X'$ by:
\[ X' = \diag(\pi(A))^{-1} X \diag(\pi(A))^{-1}. \]
We show that $(t,X')$ is feasible for the SDP on the right: Note that:
\[ 
\begin{bmatrix} t & 1^T\\ 1 & X' \end{bmatrix} =
\begin{bmatrix} 1 & 0\\ 0 & \diag(\pi(A))^{-1} \end{bmatrix}
\begin{bmatrix} t & \pi(A)^T\\ \pi(A) & X \end{bmatrix}
\begin{bmatrix} 1 & 0\\ 0 & \diag(\pi(A))^{-1} \end{bmatrix} \succeq 0 \]
Second we clearly have $X'_{ij,ij} = A_{ij}^{-2} X_{ij,ij} = 1$. Finally constraints (a') and (b') are also clearly true. Thus this shows that $(t,X')$ is feasible for the SDP of $\bartheta(\RG(A))$ and thus $\bartheta(\RG(A)) \leq t = \tau_+^{\sos}(A)$.
\end{enumerate}
\if\mapr1\qed\fi
\end{proof}

Figure \ref{fig:summaryineqs} summarizes the different quantities discussed in this section and how they relate to the quantities $\tau_+(A)$ and $\tau_+^{\sos}(A)$:

\begin{figure}[ht]
\centering
\begin{equation*}
\begingroup
\normalsize
\renewcommand*{\arraystretch}{1.5}
\begin{array}{rcccccl}
                          &      & \tau_+^{\sos}(A) & \leq & \tau_+(A) & \leq & \rank_+(A)\\
                          &      &  \rotatebox[origin=c]{270}{{\Large  $\geq$}} &  &    \rotatebox[origin=c]{270}{{\Large  $\geq$}} &  &  \quad \rotatebox[origin=c]{270}{{\Large  $\geq$}} \\
\fool(A) = \omega(\RG(A)) & \leq & \bartheta(\RG(A)) & \leq & \chi_{\fra}(\RG(A)) & \leq & \chi(\RG(A)) = \rank_B(A)
\end{array}
\endgroup
\end{equation*}
\caption{Summary of the relations between $\tau_+(A), \tau_+^{\sos}(A)$ and some combinatorial lower bounds on $\rank_+(A)$.}
\label{fig:summaryineqs}
\end{figure}

\subsection{Comparison with norm-based lower bounds on nonnegative rank}
\label{sec:normbased}

In this section we consider a class of lower bounds on the nonnegative rank that are based on homogeneous functions and which are similar to the ones proposed in \cite{fawzi2012new} or implicitly in \cite{braun2012approximation} (see also Lemma 4 in \cite{rothvoss2013matching}). We then explore their connection with the quantity $\tau_+(A)$ and we show that such lower bounds are always dominated by $\tau_+(A)$.

\begin{definition}
A function $\pn:\RR^{m\times n}_+\rightarrow \RR_+$ is called \emph{positively homogeneous} if it satisfies $\pn(\lambda A) = \lambda \pn(A)$ for any $A \in \RR^{m\times n}_+$ and $\lambda \geq 0$. Furthermore, it is called \emph{monotone} if it satisfies, for any $A,B \in \RR^{m\times n}_+$:
\[ A \leq B \; \Rightarrow \; \pn(A) \leq \pn(B), \]
where $A\leq B$ is the componentwise inequality.
\end{definition}

Norms on $\RR^{m\times n}$ form a natural class of positively homogeneous functions. Many norms also satisfy the monotonicity property, like for example, the Frobenius norm (i.e., the $\ell_2$ entrywise norm):
\[ \|A\|_F = \sqrt{\sum_{i=1}^m \sum_{j=1}^n A_{i,j}^2}, \]
or the $\ell_{\infty}$ entrywise norm:
\[ \|A\|_{\infty} = \max_{\substack{1 \leq i \leq m\\ 1\leq j \leq n}} |A_{i,j}|. \]
Define $\cA_{\pn}$ to be the set of rank-one matrices in the ``unit ball'' of $\pn$, i.e.,
\begin{equation}
\label{eq:AN}
 \cA_{\pn} := \{ X \in \RR^{m\times n}_+ \; : \; \rank X \leq 1 \text{ and } \pn(X) \leq 1\}.
\end{equation}
We can also define:
\begin{equation}
 \label{eq:pnsdef}
 \begin{aligned}
   \pns(A) &= \min \{ t > 0 \; : \; A \in t \conv(\cA_{\pn}) \}\\
           &= \max \{ L(A) \; : \; L \text{ linear} \text{ and  $L(X) \leq 1 \; \forall X \in \cA_{\pn}$} \}.
 \end{aligned}
\end{equation}
The fact that the two formulations of $\pns(A)$ above are equal follows from convex duality and the same arguments used in Lemma \ref{lem:duality}.
The following proposition shows that one can obtain a lower bound on $\rank_+(A)$ using $\pns(A)$ and $\pn(A)$:
\begin{proposition}
\label{prop:phflbrank+}
Let $\pn:\RR^{m\times n}_+ \rightarrow \RR_+$ be a monotone positively homogeneous function, and let $\pns$ be defined as in Equation \eqref{eq:pnsdef}. Then for any $A \in \RR^{m\times n}_+$, we have:
\[ \rank_+(A) \geq \frac{\pns(A)}{\pn(A)}. \]
\end{proposition}
\begin{proof}
Let $A = \sum_{i=1}^r A_i$ be a decomposition of $A$ with $r=\rank_+(A)$ terms and where each $A_i$ is rank-one and nonnegative.
Let $L$ be the optimal solution in the maximization problem of Equation \eqref{eq:pnsdef}. Then we have:
\[ \pns(A) = L(A) = \sum_{i=1}^r L(A_i) = \sum_{i=1}^r \pn(A_i) L\left(\frac{1}{\pn(A_i)} A_i\right) \overset{(a)}{\leq} \sum_{i=1}^r \pn(A_i) \overset{(b)}{\leq} \sum_{i=1}^r \pn(A) = r \pn(A) \]
where in $(a)$ we used the homogeneity of $\pn$ and the fact that $L(X) \leq 1$ when $\pn(X) \leq 1$, and in $(b)$ we used the fact that for each $i$ we have $A_i \leq A$, and thus by monotonicity of $\pn$ we have $\pn(A_i) \leq \pn(A)$. Thus we finally get that
\[ r \geq \frac{\pns(A)}{\pn(A)} \]
which is what we wanted.
\if\mapr1\qed\fi
\end{proof}

In \cite{fawzi2012new} the authors studied the case where $\pn$ is the Frobenius norm, and where the associated quantity $\pns$ was called the \emph{nonnegative nuclear norm} and was denoted by $\nu_+$. For this particular choice of $\pn$ the following stronger lower bound was shown to hold:
\[ \rank_+(A) \geq \left(\frac{\pns(A)}{\pn(A)}\right)^2. \]
Also in \cite{rothvoss2013matching} the lower bound corresponding to $\pn = \|\cdot \|_{\infty}$ (entry-wise infinity norm) was used to obtain exponential lower bounds on the nonnegative rank of a certain matrix of interest in extended formulations of polytopes (the slack matrix associated with the matching polytope).

In the next theorem we show that any lower bound on $\rank_+$ obtained from monotone positively homogeneous functions like in Proposition \ref{prop:phflbrank+} is always dominated by $\tau_+(A)$.

\begin{theorem}
Let $\pn:\RR^{m\times n}_+\rightarrow \RR_+$ be a monotone positively homogeneous function, and let $\pns$ be as defined in Equation \eqref{eq:pnsdef}. Then for any $A \in \RR^{m \times n}_+$ we have:
\[ \rank_+(A) \geq \tau_+(A) \geq \frac{\pns(A)}{\pn(A)}. \]
\end{theorem}
\begin{proof}
First note that we have the inclusion 
\begin{equation}
 \label{eq:inclusionA+An}
 \frac{1}{\pn(A)} \cA_{+}(A) \subseteq \cA_{\pn}.
\end{equation}
Indeed if $R$ is rank-one and satisfies $0\leq R \leq A$ then we have, by homogeneity and monotonicity of $\pn$, 
\[ \pn\left(\frac{1}{\pn(A)} R\right) =  \frac{1}{\pn(A)}\pn(R) \leq \frac{1}{\pn(A)} \pn(A) \leq 1. \]

Let $L$ be the optimal linear form in the definition of $\pns(A)$ in \eqref{eq:pnsdef}. Since $L \leq 1$ on $\cA_{\pn}$, by the inclusion \eqref{eq:inclusionA+An} we have that $L \leq 1$ on $\frac{1}{\pn(A)} \cA_+(A)$ or equivalently that $\frac{1}{\pn(A)} L \leq 1$ on $\cA_+(A)$. Thus by definition of $\tau_+(A)$ we have 
\[ \tau_+(A) \geq \frac{1}{\pn(A)} L(A) = \frac{\pns(A)}{\pn(A)}. \]
\if\mapr1\qed\fi
\end{proof}

We now show that the quantity $\tau_+(A)$ actually fits in the class of lower bounds of Proposition \ref{prop:phflbrank+}, where the homogeneous function $\pn$ \emph{depends on $A$}. Specifically if $A$ is a nonnegative matrix, we can define $\pn_A$ as follows:
\[ \pn_A(X) = \min \{ t > 0 \; : \; X \leq tA \}. \]
Clearly $\pn_A$ is a monotone positively homogeneous function, and it satisfies $\pn_A(A) = 1$. Note that the set of atoms $\cA_{\pn_A}$ associated to $\pn_A$ (cf. Equation \eqref{eq:AN}) is nothing but $\cA_+(A)$. Thus it follows directly from the definition \eqref{eq:pnsdef} of $\pns(A)$ that $\pns_A(A) = \tau_+(A)$.
To summarize we can write that:
\[ \tau_+(A) = \sup_{\substack{\pn \text{ monotone and} \\ \text{positively homogeneous}}} \frac{\pns(A)}{\pn(A)}. \]
\subsection{Examples}
In this section we apply the lower bounds $\tau_+$ and $\tau_+^{\sos}$ to some examples of matrices. We first derive an explicit formula for the lower bounds for $2\times 2$ matrices and then we consider a toy example drawing from the geometric interpretation of the nonnegative rank.

\subsubsection{$2\times 2$ matrices}

When $A$ is a $2\times 2$ nonnegative matrix, we can get a closed-form formula for the values of $\tau_+(A)$ and $\tau_+^{\sos}(A)$:
\begin{proposition}
For a $2\times 2$ nonnegative matrix
\[ A = \begin{bmatrix} x & y\\ z & w \end{bmatrix} \]
we have (when $xw+yz\neq 0$):
\begin{equation}
\label{eq:tau+-formula2x2}
 \tau_{+}(A) = \begin{cases}
2 - \displaystyle\frac{xw}{yz} & \text{ \upshape if } \;\; xw \leq yz\\[0.5cm]
2 - \displaystyle\frac{yz}{xw} & \text{ \upshape if } \;\; xw \geq yz
\end{cases}
\quad
\text{ and }
\quad
 \tau_{+}^{\sos}(A) = \begin{cases}
\displaystyle\frac{2}{1+\frac{xw}{yz}} & \text{ \upshape if } \;\; xw \leq yz\\[0.5cm]
\displaystyle\frac{2}{1+\frac{yz}{xw}} & \text{ \upshape if } \;\; xw \geq yz
\end{cases}
\end{equation}
If $xw=yz=0$ and $A$ has at least one positive entry then $\tau_+^{\sos}(A) = 1$.
\end{proposition}
\begin{proof}
Assume first $x,y,z,w > 0$. By the diagonal invariance property we have:
\[ \tau_{+}(A)
= \tau_{+} \left(\begin{bmatrix} 1 & 0 \\ 0 & z^{-1} x \end{bmatrix} \begin{bmatrix} x & y\\ z & w \end{bmatrix} \begin{bmatrix} x^{-1} & 0 \\ 0 & y^{-1} \end{bmatrix}\right)
= \tau_{+} \left(\begin{bmatrix} 1 & 1\\ 1 & \frac{xw}{yz} \end{bmatrix}\right), \]
and the same is true for $\tau_+^{\sos}$. To prove the formula we thus only need to study matrices of the form $\left[\begin{smallmatrix} 1 & 1\\ 1 & e\end{smallmatrix}\right]$.
The following lemma gives the value of $\tau_+$ and $\tau_+^{\sos}$ of such matrices as a function of $e$:
\begin{lemma}
\label{lem:tau+sos_2x2}
Let
\[ A(e) = \begin{bmatrix} 1 & 1\\ 1 & e \end{bmatrix}. \]
Then 
\begin{equation}
\label{eq:tau+e}
\tau_+(A(e)) = \begin{cases} 2-e & \text{ if } 0 \leq e \leq 1\\
2 - 1/e & \text{ if } e \geq 1
\end{cases}
\quad
\text{ and }
\quad
\tau_{+}^{\sos}(A(e)) = \begin{cases} \frac{2}{1+e} & \text{ if } 0 \leq e \leq 1\\
\frac{2}{1+1/e} & \text{ if } e \geq 1
\end{cases}.
\end{equation}
\end{lemma}
Figure \ref{fig:plot_tau_sos_2x2_matrices} below illustrates the result of Lemma \ref{lem:tau+sos_2x2} and shows the graphs of $\tau_+(A(e))$ and $\tau_+^{\sos}(A(e))$ as a function of $e$. Note that the functions $\tau_+$ and $\tau_+^{\sos}$ are in general not convex.
\begin{figure}[ht]
  \centering
  \includegraphics[width=9cm]{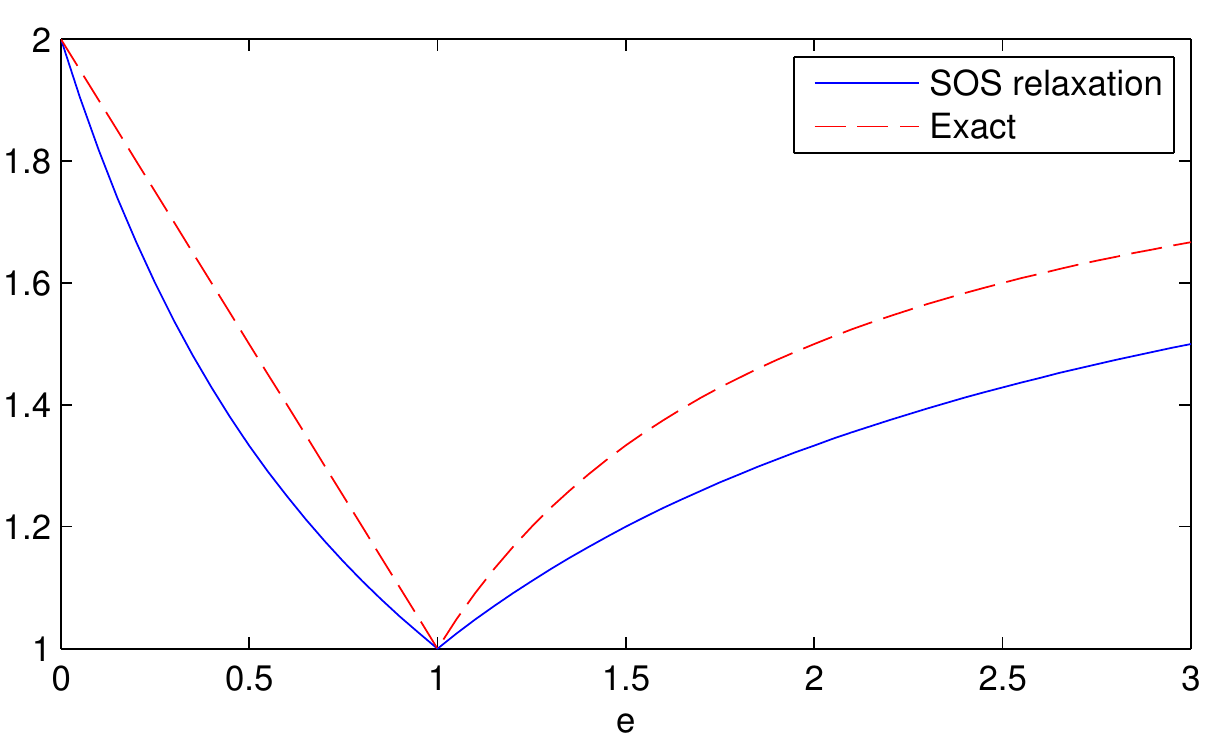}
  \caption{Graph of the functions $\tau_+(A(e))$ and $\tau_+^{\sos}(A(e))$ as a function of $e$ (cf. Equation \eqref{eq:tau+e} for the expressions).}
   \label{fig:plot_tau_sos_2x2_matrices}
\end{figure}
\begin{proof}[Proof of Lemma \ref{lem:tau+sos_2x2}]
First observe that it suffices to look at matrices $A(e)$ for $0 \leq e\leq 1$. Indeed for $e \geq 1$, the matrix $A(e)$ can be obtained from the matrix $A(1/e)$ by permuting the two columns then scaling the second row by $e$, namely we have:
\[ \begin{bmatrix} 1 & 1\\ 1 & e \end{bmatrix}
=
\begin{bmatrix} 1 & 0\\ 0 & e \end{bmatrix}
\begin{bmatrix} 1 & 1\\ 1 & 1/e \end{bmatrix} \begin{bmatrix} 0 & 1\\ 1 & 0 \end{bmatrix}.
\]
Thus by Theorem \ref{thm:tau+properties} on the properties of $\tau_+$ and $\tau_+^{\sos}$ we have:
\[ \tau_+(A(e)) = \tau_+(A(1/e)) \quad \text{ and } \quad \tau_+^{\sos}(A(e)) = \tau_+^{\sos}(A(1/e)). \]
In the following we thus fix $e \in [0,1]$ and $A=A(e)$ and we show that $\tau_+(A) = 2-e$ and $\tau_+^{\sos}(A) = 2/(1-e)$.

\noindent $\bullet$ \; We first show that $\tau_+(A) = 2-e$. To prove that $\tau_+(A) \geq 2-e$ we exhibit a linear function $L$ such that $L(R) \leq 1$ for all $R \in \cA_+(A)$. Define $L$ by:
\[ L\left ( \begin{bmatrix} a & b\\ c &d\end{bmatrix} \right) = b+c - e\cdot a. \]
The next lemma shows that $L(R) \leq 1$ for all $R \in \cA_+(A)$:
\begin{lemma}
For any $0\leq (a,b,c,d) \leq (1,1,1,e)$ such that $ad=bc$ we have 
\[ b+c-e\cdot a\leq 1. \]
\end{lemma}
\begin{proof}
If $d = 0$ then either $b=0$ or $c=0$ and the inequality is true because $b,c\leq 1$. Now if $d > 0$ we have:
\[
b+c - e\cdot a = b+c-e \frac{bc}{d} \overset{(*)}{\leq} b+c-bc = b\cdot 1 + (1-b) \cdot c \leq \max(1,c) \leq 1
\]
where in (*) we used the fact that $e/d\geq 1$.
\if\mapr1\qed\fi
\end{proof}
For this choice of $L$ we have $L(A) = 1 + 1 - e = 2-e$. Thus this shows that $\tau_+(A) \geq 2-e$.

We now show that $\tau_+(A) \leq 2-e$. For this we prove that $A \in (2-e) \conv(\cA_+(A))$. We have the following decomposition of $\frac{1}{2-e} A$:
\[ \frac{1}{2-e} \begin{bmatrix} 1 & 1\\ 1 & e\end{bmatrix} \; = \; \lambda \begin{bmatrix} 1 & 1\\ e & e\end{bmatrix} + \lambda \begin{bmatrix} 1 & e\\ 1 & e\end{bmatrix} + \mu \begin{bmatrix} 0 & 1\\ 0 & 0\end{bmatrix} + \mu\begin{bmatrix}0 & 0\\ 1 & 0\end{bmatrix} \]
where
\[ \lambda = \frac{1}{2(2-e)} \quad \text{ and } \quad \mu = \frac{1-e}{2(2-e)}. \]
Note that $\lambda,\mu \geq 0$, $2\lambda + 2\mu=1$ and that the four $2\times 2$ matrices in the decomposition belong to $\cA_+(A)$. Thus this shows that $A \in (2-e) \conv(\cA_+(A))$ and thus $\tau_+(A) \leq 2-e$.

\medskip

\noindent $\bullet$ \; We now look at the SDP relaxation $\tau_+^{\sos}$ and we show that $\tau_+^{\sos}(A(e)) = 2/(1+e)$. To do so, we exhibit primal and dual feasible points for the semidefinite programs \eqref{eq:tau+sos_min} and \eqref{eq:tau+sos_max} which attain the value $2/(1+e)$. These are shown in the table below:
\[
\begin{array}{c|c}
\text{\it Primal (SDP \eqref{eq:tau+sos_min})} & \text{\it Dual (SOS program \eqref{eq:tau+sos_max})}\\ \hline
\begin{array}{l}
t = \displaystyle\frac{2}{1+e}\phantom{\Biggl(\Biggr)}\\
X = \frac{1}{t} aa^T + \frac{1-e}{2} \underbrace{\begin{bmatrix}
1 & 0 & 0 & e\\
0 & 1 & -1 & 0\\
0 & -1 & 1 & 0\\
e & 0 & 0 & e^2\end{bmatrix}}_{\tilde{X}}\\
\;\;\text{ where } a = \vec(A)
\end{array}
&
\begin{array}{l}
L(X) = \frac{1+2e}{(1+e)^2}(X_{12}+X_{21}) - \frac{e}{(1+e)^2} X_{11} - \frac{1}{(1+e)^2} X_{22}\\
SOS(X) = \left(1-\frac{1}{1+e} X_{12}-\frac{1}{1+e}X_{21}\right)^2 + \left(\frac{\sqrt{e}}{1+e} X_{11} - \frac{1}{\sqrt{e}(1+e)}X_{22}\right)^2\\
D_{11} = \frac{e}{(1+e)^2} \qquad D_{12} = \frac{1}{(1+e)^2}\\
D_{21} = \frac{1}{(1+e)^2} \qquad D_{22} = \frac{1}{e(1+e)^2}\\
\nu = \frac{2}{(1+e)^2}
\end{array}
\end{array}
\]
It is not hard to show that these are feasible points: For the primal SDP we have to verify that the matrix
\[ \begin{bmatrix} t & a^T \\ a & X \end{bmatrix} \]
is positive semidefinite.
By Schur complement theorem this is equivalent to having $X \succeq aa^T / t$. This is true for the matrix $X$ defined above since by definition $X = aa^T / t + (1-e)/2 \cdot \tilde{X}$ where $\tilde{X} \succeq 0$. Also one can easily verify that for any $i,j$ we have $X_{ij,ij} = A_{ij}^2$. Finally the $2\times 2$ minor constraint is satisfied because we have:
\[ X_{(1,1),(2,2)} = X_{1,4} = \frac{1+e}{2} e + \frac{1-e}{2} e = e \]
and
\[ X_{(1,2),(2,1)} = X_{2,3} = \frac{1+e}{2} - \frac{1-e}{2} = e \]
and thus $X_{(1,1),(2,2)} = X_{(1,2),(2,1)}$.\\
To verify that the SOS certificate is valid we need to verify that the following identity holds:
\[ 1-L(X) = SOS(X) + \sum_{i,j} D_{ij} X_{ij} (A_{ij} - X_{ij}) + \nu (X_{11}X_{22} - X_{12}X_{21}), \]
which can be easily verified. Also the objective value is:
\[ L(A) = \frac{1+2e}{(1+e)^2}\cdot 2 - \frac{e}{(1+e)^2} \cdot 1 - \frac{1}{(1+e)^2}\cdot e = \frac{2}{1+e}. \]
\if\mapr1\qed\fi
\end{proof}
The lemma above together with the diagonal invariance property proves the formula \eqref{eq:tau+-formula2x2} when all entries $x,y,z,w$ are strictly positive. It remains to consider the case where some entries are equal to zero:
\begin{itemize}
\item If $xw = 0$ and $yz > 0$ we have to show that $\tau_+^{\sos}(A) = 2$.  To show this, observe that in this case $A$ has a fooling set of size 2 (i.e., the clique number of $\RG(A)$ is 2), and thus by Theorem \ref{thm:relation-tau-RGA} we have $\tau_+(A) \geq \tau_+^{\sos}(A) \geq 2$.\\
The case $xw > 0$ and $yz = 0$ is identical.
\item Otherwise, we have $xw=yz=0$. In this case the fooling set number is 1 and the nonnegative rank is 1 and thus $\tau_+^{\sos}(A) = 1$.
\end{itemize}
\if\mapr1\qed\fi
\end{proof}

\subsubsection{Nested rectangles problem}

We consider in this section an example dealing with the geometric interpretation of the nonnegative rank. The problem of finding a nonnegative factorization of a matrix $A$ is related to the problem of nested polytopes in geometry, where one looks for a polytope $T$ with minimal number of vertices that is sandwiched between two given polytopes $S$ and $P$, i.e., $S \subseteq T \subseteq P$. In this section we briefly review this geometric interpretation of the nonnegative rank and we then explore a simple example drawing from this interpretation. 

\paragraph{Geometric interpretation of the nonnegative rank} Consider a nonnegative matrix $A \in \RR^{m\times n}_+$ with columns $a_1,\dots,a_n \in \RR^m_+$ and assume for simplicity that for all $i$, $a_i \in \Delta^{m-1}$ where $\Delta^{m-1}$ is the unit simplex:
\[ \Delta^{m-1} = \{ x \in \RR^m \; : \; 1^T x=1, x \geq 0 \}. \]
Let $S$ be the polytope formed as the convex hull of the $a_i$'s:
\[ S = \conv(a_1,\dots,a_n). \]
Assume that we have a nonnegative factorization $A=UV$ of $A$. Using an appropriate diagonal scaling $U\leftarrow UD$ and $V\leftarrow D^{-1} V$, we can assume that the columns of $U$ and $V$ are in the unit simplex, i.e., $1^T U = 1^T$ and $1^T V = 1^T$. Geometrically, the factorization $A=UV$ simply says that each column of $A$ is equal to a convex combination of the columns of $U$, where the coefficients are given by the entries of $V$. Furthermore the columns of $U$ are all in the unit simplex. Thus if we let $T = \conv(u_1,\dots,u_k)$ we have the following inclusion:
\begin{equation}
 \label{eq:sandwiched-polytopes}
S \subseteq T \subseteq \Delta^{m-1}.
\end{equation}It is not too difficult to see that the nonnegative rank of $A$ is actually the smallest $k$ such that we can find a polytope $T$ with $k$ vertices that satisfies \eqref{eq:sandwiched-polytopes}.

When the matrix $A$ is $4\times 4$ and has rank 3, this sandwiched polytope problem can be reduced to a problem in the plane since one can show that it is sufficient to work in the two-dimensional affine subspace spanned by $a_1,\dots,a_4$. Note however that this is not true in general, i.e., we cannot in general reduce the problem to $\aff(A)$, and this in fact is the main difference between the nonnegative rank and the \emph{restricted nonnegative rank} defined in \cite{gillis2012geometric}. The following proposition summarizes the geometric interpretation of the nonnegative rank for $4\times 4$ matrices of rank 3.
\begin{proposition}
\label{prop:geomrank+}
Let $A$ be a $4 \times 4$ nonnegative matrix of rank 3 and assume that the columns $a_1,\dots,a_4$ of $A$ satisfy $1^T a_i = 1$ for all $i=1,\dots,4$. Let $\aff(A)$ be the affine hull of $a_1,\dots,a_4$ and note that $\aff(A)$ is a two-dimensional affine subspace of $\RR^4$. The nonnegative rank of $A$ is the smallest integer $k$ such that there exists a polytope $T \subseteq \aff(A)$ with $k$ vertices such that
\[ S \subseteq T \subseteq P \]
where $S$ and $P$ are two polytopes that live in $\aff(A)$ and that are defined by:
\[ S = \conv(a_1,\dots,a_4) \quad \text{ and } \quad P = \Delta^3 \cap \aff(A). \]
\end{proposition}

For more details on the geometric interpretation of the nonnegative rank we refer the reader to \cite{gillis2012geometric} and \cite{mond2003stochastic}.

\paragraph{Example} Let $P = [-1,1]^2$ be the square in the plane, and let 
\[ S(a,b) = [-a,a]\times [-b,b] \] be the rectangle of dimensions $2a\times 2b$ centered at $(0,0)$. We consider the following question: Does there exist a triangle $T$ contained in $P$ and that contains $S(a,b)$?

\begin{figure}[ht]
  \centering
  \includegraphics[width=9cm]{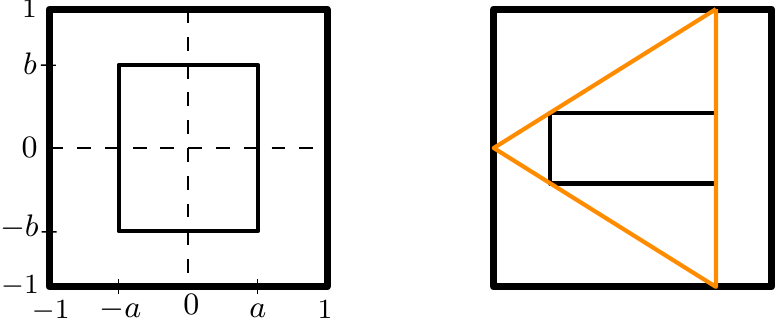}
  \caption{Left: The square $P=[-1,1]^2$ and the rectangle $S(a,b) = [-a,a]\times [-b,b]$. Right: An instance where there exists a triangle $T$ such that $S(a,b) \subset T \subset P$.}
  \label{fig:figure_nested_rectangles}
\end{figure}

 Using the geometric interpretation of the nonnegative rank, one can show that such a triangle exists if, and only if, the nonnegative rank of the following $4\times 4$ matrix is equal to 3:
\[
M(a,b) =
\begin{bmatrix}
1-a & 1+a & 1+a & 1-a\\
1-b & 1-b & 1+b & 1+b\\
1+a & 1-a & 1-a & 1+a\\
1+b & 1+b & 1-b & 1-b
\end{bmatrix}.
\]
In fact the matrix $M(a,b)$ is constructed in such a way that the polytopes $S$ and $P$ in Proposition \ref{prop:geomrank+} are respectively $S=S(a,b)$ and $P=[-1,1]^2$.

By computing the quantity $\tau_+^{\sos}(M(a,b))$ we can certify the \emph{non-existence} of such a triangle if we get $\tau_+^{\sos}(M(a,b)) > 3$. We have computed the value of $\tau_+^{\sos}(M(a,b))$ numerically for a grid of values in $(a,b) \in [0,1]^2$ and Figure \ref{fig:nested_squares_plot} shows the region where we got $\tau_+^{\sos}(M(a,b)) \geq 3$.
\begin{figure}[ht]
\centering
\includegraphics[width=6cm]{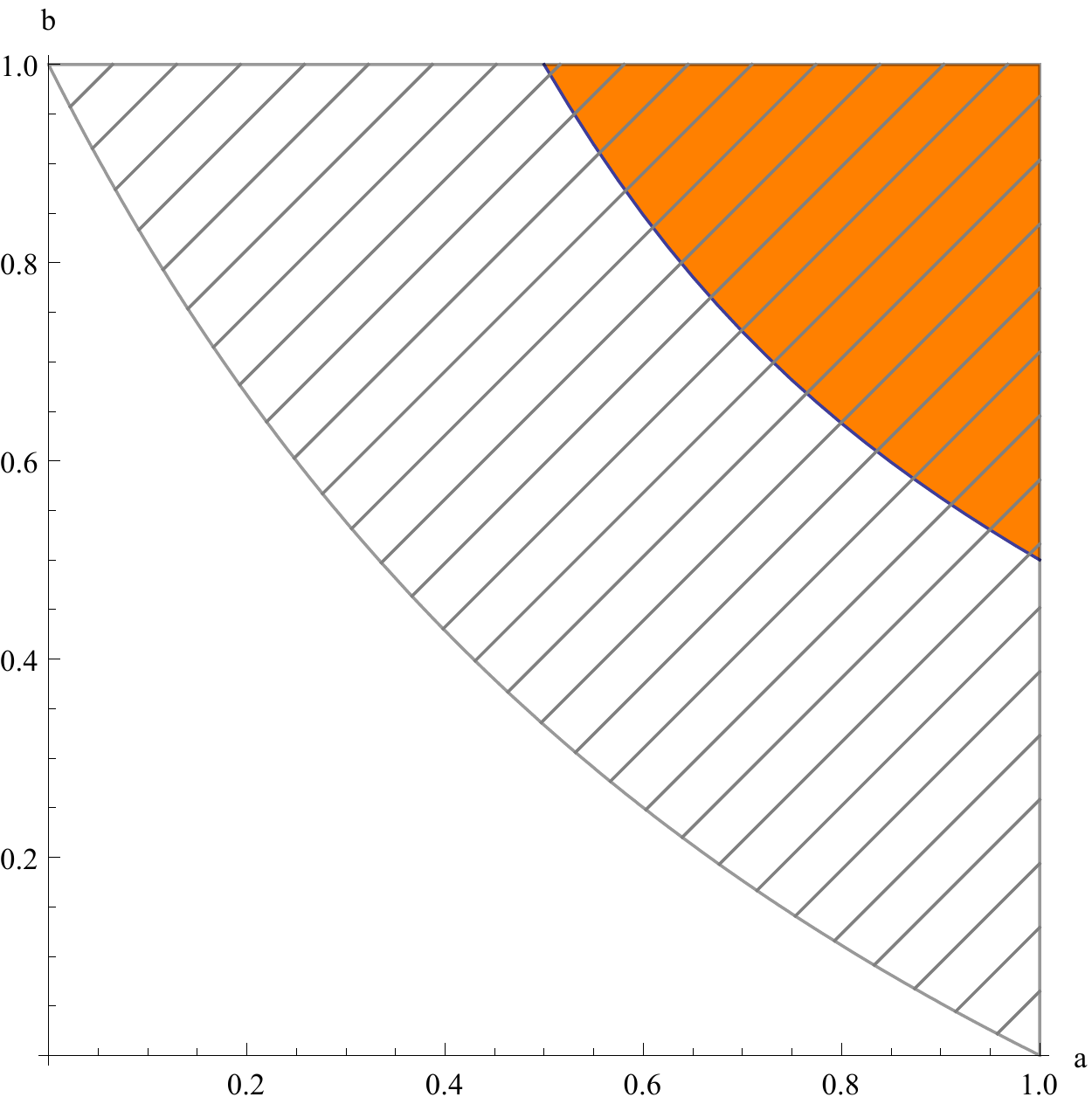}
\caption{The colored region is the region where $\tau_+^{\sos}(M(a,b)) \geq 3$, and in dashed is the correct region $\{(a,b) : \rank_+(M(a,b)) = 4\}$ (cf. proposition \ref{prop:existenceT}).}
\label{fig:nested_squares_plot}
\end{figure}

Using geometric considerations, one can actually solve the problem analytically and show that a triangle $T$ exists with $S(a,b) \subset T \subset P$ if, and only if, $(1+a)(1+b) \leq 2$:
\begin{proposition}
\label{prop:existenceT}
There exists a triangle $T$ such that $S(a,b) \subset T \subset P$ if, and only if, $(1+a)(1+b) \leq 2$.
\end{proposition}
\begin{proof}[Sketch of proof]
\begin{itemize}
\item We first prove the direction $\Rightarrow$: Let $T$ be a triangle such that $S(a,b) \subset T \subset P$. We can clearly assume that the vertices of $T$ are all on the boundary of $P$.  Then since $T$ has only 3 vertices, there is one edge of $P$ that $T$ does not touch \footnote{We assume here that the vertices of $T$ do not coincide with any vertex of $P$---i.e., they all lie proper on the edges of $P$. Note that if one vertex of $T$ coincides with a vertex of $P$, then the other two vertices of $T$ can be easily determined from the inner rectangle and one can show that the triangle will indeed contain the inner rectangle if and only if $(1+a)(1+b) \leq 2$.}. Assume without loss of generality that the edge that $T$ misses is the edge joining $(1,-1)$ to $(1,1)$. In this case $T$ has the form shown in Figure \ref{fig:nested_squares_form1}(a) below.

It is easy to see that one can move the vertices of $T$ so that it has the ``canonical'' form shown in Figure \ref{fig:nested_squares_form1}(b) while still satisfying $S(a,b) \subset T \subset P$. The coordinates of the vertices of the canonical triangle are respectively $(a,-1), (a,1)$ and $(-1,0)$. Using a simple calculation one can show that this triangle contains $S(a,b)$ if and only if $(1+a)(1+b) \leq 2$.

\begin{figure}[ht]
  \centering
  \includegraphics[width=6cm]{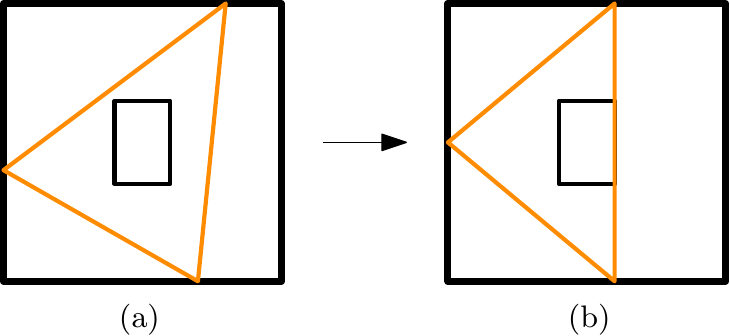}
  \caption{(a) A triangle $T$ such that $S(a,b) \subset T \subset P$. (b) Canonical form of triangle, where one side is parallel to the axis and touches the corresponding of $S(a,b)$.}
  \label{fig:nested_squares_form1}
\end{figure}

\item To show that the condition $(1+a)(1+b) \leq 2$ is sufficient, it suffices to consider the triangle of the form depicted in Figure \ref{fig:nested_squares_form1}(b) which can be shown to contain $S(a,b)$ if $(1+a)(1+b) \leq 2$.
\end{itemize}
\if\mapr1\qed\fi
\end{proof}

\section{Nonnegative rank of tensors}
\label{sec:rank+tensors}

\subsection{Definitions}

In this section we show how the same lower bound technique can be used to obtain lower bounds on nonnegative tensor rank.

Let $A = [a_{i_1\dots i_n}]$ be a nonnegative tensor of size $d_1\times \dots \times d_n$. A tensor of rank one is a tensor of the form:
\[ u_1\otimes \dots \otimes u_n \in \RR^{d_1\times \dots \times d_n} \]
where $u_1 \in \RR^{d_1}, \dots, u_n \in \RR^{d_n}$ and
\[ (u_1 \otimes \dots \otimes u_n)_{i_1,\dots,i_n} \overset{\text{def}}{=} (u_1)_{i_1} \dots (u_n)_{i_n}. \]
The nonnegative rank of $A$, denoted $\rank_+(A)$ is the smallest integer $r$ for which $A$ can be written as the sum of $r$ nonnegative rank-one tensors. Observe that in any rank-one nonnegative decomposition of $A$:
\[ A = \sum_{i=1}^r A_i \]
where $A_i$ are rank-one nonnegative tensors, each $A_i$ must satisfy $0 \leq A_i \leq A$ (componentwise inequality). Like in the matrix case, this motivates the definition of:
\[ \cA_+(A) = \Bigl\{ R \in \RR^{d_1\times \dots \times d_n} \; : \; \rank R \leq 1 \; \text{ and } \; 0\leq R \leq A \Bigr\}. \]
We can then define $\tau_+(A)$ also in the same way as for matrices:
\[
\begingroup
\renewcommand*{\arraystretch}{2.0}
\begin{array}{rl}
\tau_+(A) &= \;\;\min \quad\; t \;\quad\quad \text{ subject to } \quad A \in t \conv(\cA_+(A))\\
          &= \displaystyle\max_{L \text{ linear}} L(A) \quad \text{ subject to } \quad L(R) \leq 1 \; \forall R \in \cA_+(A).
\end{array}
\endgroup
\]
The quantity $\tau_+(A)$ then verifies:
\[ \tau_+(A) \leq \rank_+(A). \]

\subsection{Semidefinite programming relaxation}

To obtain a semidefinite programming relaxation of $\tau_+(A)$ for tensors, we use the same procedure as in the case of matrices, described in section \ref{sec:tau+sdprelaxation}. We construct an over-relaxation of $\conv(\cA_+(A))$ which can be represented using linear matrix inequalities. The variety of rank-one tensors $R \in \RR^{d_1\times \dots \times d_n}$ is known as the \emph{$n$-factor Segre variety} and can be described using quadratic equations in the entries of $R$; those equations are described in \cite{grone1977decomposable} and can be summarized as follows: If $\mb{i}=(i_1,\dots,i_n) \in [d_1]\times \dots \times [d_n]$ and $\mb{j}=(j_1,\dots,j_n) \in [d_1]\times \dots \times [d_n]$, and $k \in \{1,\dots,n\}$, define $\swapix{\mb{i}}{\mb{j}}{k}$ the multi-index obtained from $\mb{i}$ by replacing the $k$'th position by $j_k$, i.e.,
\[ \swapix{\mb{i}}{\mb{j}}{k} = (i_1,\dots,j_k,\dots,i_n). \]
The following gives the characterization of rank-one tensors using quadratic equations \cite{grone1977decomposable}:
\begin{equation}
\label{eq:tensor-rank1-equations}
\begingroup
\renewcommand*{\arraystretch}{1.3}
\begin{array}{c}
\rank R \leq 1 \\
\Longleftrightarrow\\
\forall \mb{i}, \mb{j} \in [d_1]\times \dots \times [d_n], \;\; \forall k=1\dots,n: \;\; R[\mb{i}]\cdot R[\mb{j}] = R[\swapix{\mb{i}}{\mb{j}}{k}]\cdot R[\swapix{\mb{j}}{\mb{i}}{k}].
\end{array}
\endgroup
\end{equation}
If we introduce $r = \vec(R)$, then the equations above are linear in the entries of the matrix $X=rr^T$:
\begin{equation}
\label{eq:rankone-equations-tensor-linearized}
\forall \mb{i}, \mb{j} \in [d_1]\times \dots \times [d_n], \;\; \forall k=1\dots,n: \;\; X[\mb{i},\mb{j}] = X[\swapix{\mb{i}}{\mb{j}}{k},\swapix{\mb{j}}{\mb{i}}{k}]
\end{equation}
 Furthermore, the constraint $R \leq A$ implies that:
\[ (rr^T)_{\mb{i},\mb{i}} \leq r_{\mb{i}} A_{\mb{i}}, \]
for any $\mb{i}=(i_1,\dots,i_n) \in [d_1]\times \dots \times [d_n]$, 
and these are linear inequalities in the entries of the matrix
\[
\begin{bmatrix} 1\\ r \end{bmatrix} \begin{bmatrix} 1 \\ r \end{bmatrix}^T = \begin{bmatrix} 1 & r^T \\ r & rr^T \end{bmatrix}.
\]
Using these observations, we obtain the following over-relaxation of $\conv(\cA_+(A))$:
\begin{equation}
 \label{eq:overrelaxationA-tensors}
 \conv(\cA_+(A)) \subseteq \cA_+^{\sos}(A)
\end{equation}
where
\begin{equation}
\label{eq:convAsos-tensors}
\begin{aligned}
\cA_+^{\sos}(A) = \Biggl\{ R \in \RR^{d_1\times \dots \times d_n} \; : \; & \exists X \in \S^{d_1\dots d_n} \; \text{ such that } \; \begin{bmatrix}
1 & \vec(R)^T\\
\vec(R) & X
\end{bmatrix} \succeq 0\\
& \qquad \text{ and } X_{\mb{i},\mb{i}} \leq R_{\mb{i}} A_{\mb{i}} \quad \forall \mb{i}=(i_1,\dots,i_n) \in [d_1]\times \dots \times [d_n]\\
& \qquad \text{ and $X$ satisfies rank-one equations \eqref{eq:rankone-equations-tensor-linearized}} \Biggr\}.
\end{aligned}
\end{equation}
 The semidefinite programming relaxation of $\tau_+(A)$ can thus be defined as follows:
\[ \tau_+^{\sos}(A) = \min \{ t > 0 \; : \; A \in t\cA^{\sos}_+(A)\}, \]
and it satisfies:
\[ \tau_+^{\sos}(A) \leq \tau_+(A) \leq \rank_+(A). \]
More explicitly, $\tau_+^{\sos}(A)$ is the solution of the following semidefinite program:
\begin{equation}
\begingroup
\renewcommand*{\arraystretch}{1.3}
\label{eq:tau+sos_min-tensors}
\begin{array}{rrl}
\tau_+^{\sos}(A) &= 
\;\;\underset{t,X}{\min} & t\\
&\text{s.t.} &
\begin{bmatrix} t & \vec(A)^T\\ \vec(A) & X \end{bmatrix} \succeq 0\\
&& X_{\mb{i},\mb{i}} \leq A_{\mb{i}}^2 \quad \forall \mb{i} \in [d_1]\times \dots \times [d_n]\\
&& \text{$X$ satisfies rank-one equations \eqref{eq:rankone-equations-tensor-linearized}}
\end{array}
\endgroup
\end{equation}
Like for the nonnegative rank of matrices (cf. Section \ref{sec:rank+sdp}), the dual of the semidefinite program \eqref{eq:tau+sos_min-tensors} can be written as the following sum-of-squares program:
\begin{equation}
\label{eq:tau+sos_max-tensors}
\tau_+^{\sos}(A) = 
\begin{array}[t]{ll}
\text{max} & L(A)\\
\text{s.t.} & L \text{ is a linear form}\\
& 1 - L(X) = SOS(X) + \displaystyle\sum_{i_1,\dots,i_n} D_{i_1 \dots i_n} X_{i_1 \dots i_n} (A_{i_1 \dots i_n} - X_{i_1 \dots i_n})\;\; \text{ mod } I\\
& D_{i_1 \dots i_n} \geq 0 \quad \forall (i_1,\dots,i_n) \in [d_1]\times \dots \times [d_n]\\
& SOS(X) \text{ is a sum-of-squares polynomial}
\end{array}
\end{equation}
Here $I$ is the ideal of rank-one tensors described by the quadratic equations in \eqref{eq:tensor-rank1-equations}.

\subsection{Properties}

The quantities $\tau_+$ and $\tau_+^{\sos}$ for tensors satisfy the same properties as for matrices. We mention here the invariance under scaling property (we omit the proof since it is very similar as for the matrix case):
\begin{theorem}
Let $A$ be a nonnegative tensor of size $d_1\times \dots \times d_n$. Let $D^1\in \RR^{d_1}_+, D^2 \in \RR^{d_2}_+, \dots , D^n \in \RR^{d_n}_+$ be nonnegative vectors with strictly positive entries and let $A'$ be the tensor defined by:
\[ A'_{i_1,\dots,i_n} = D^1_{i_1} \dots D^n_{i_n} A_{i_1,\dots,i_n}, \]
for all $(i_1,\dots,i_n) \in [d_1]\times \dots \times [d_n]$. Then
\[ \tau_+(A') = \tau_+(A) \quad \text{ and } \quad \tau_+^{\sos}(A') = \tau_+^{\sos}(A). \]
\end{theorem}

\subsection{Example}

Let $A$ be the $2\times 2\times 2$ tensor defined by:
\begin{equation}
\label{eq:tensorp}
 A =
\left[
\begin{array}{rr|rr}
x & 1 & w & 1\\
1 & w & 1 & x
\end{array}\right],
\end{equation}
where $x,w \geq 0$ (in the notation above, the first $2\times 2$ block is the slice $A(\cdot,\cdot,1)$ and the second $2\times 2$ block is the slice $A(\cdot,\cdot,2)$). Such $2\times 2\times 2$ tensors were studied in \cite[Example 2.3]{allman2013tensors}.
In the paper \cite{allman2013tensors}, necessary and sufficient conditions are given for a nonnegative tensor to have nonnegative rank $\leq 2$. For the $2\times 2\times 2$ tensor of Equation \eqref{eq:tensorp}, we get that:
\[ \rank_+ A \leq 2 \; \Longleftrightarrow \; xw \geq 1 \text{ or } x=w. \]
Figure \ref{fig:region2x2x2tensor} below shows the region where $\rank_+(A) \leq 2$ in green. We have computed the value of $\tau_+^{\sos}(A)$ numerically for a grid of values $(x,w) \in [0,3]\times [0,3]$ and we show on Figure \ref{fig:region2x2x2tensor} the region where $\tau_+^{\sos}(A) > 2$. The region in white in the figure correspond to tensors $A$ where $\tau_+^{\sos}(A) \leq 2 < \rank_+(A)$.
\begin{figure}[ht]
\centering
\includegraphics[width=10cm]{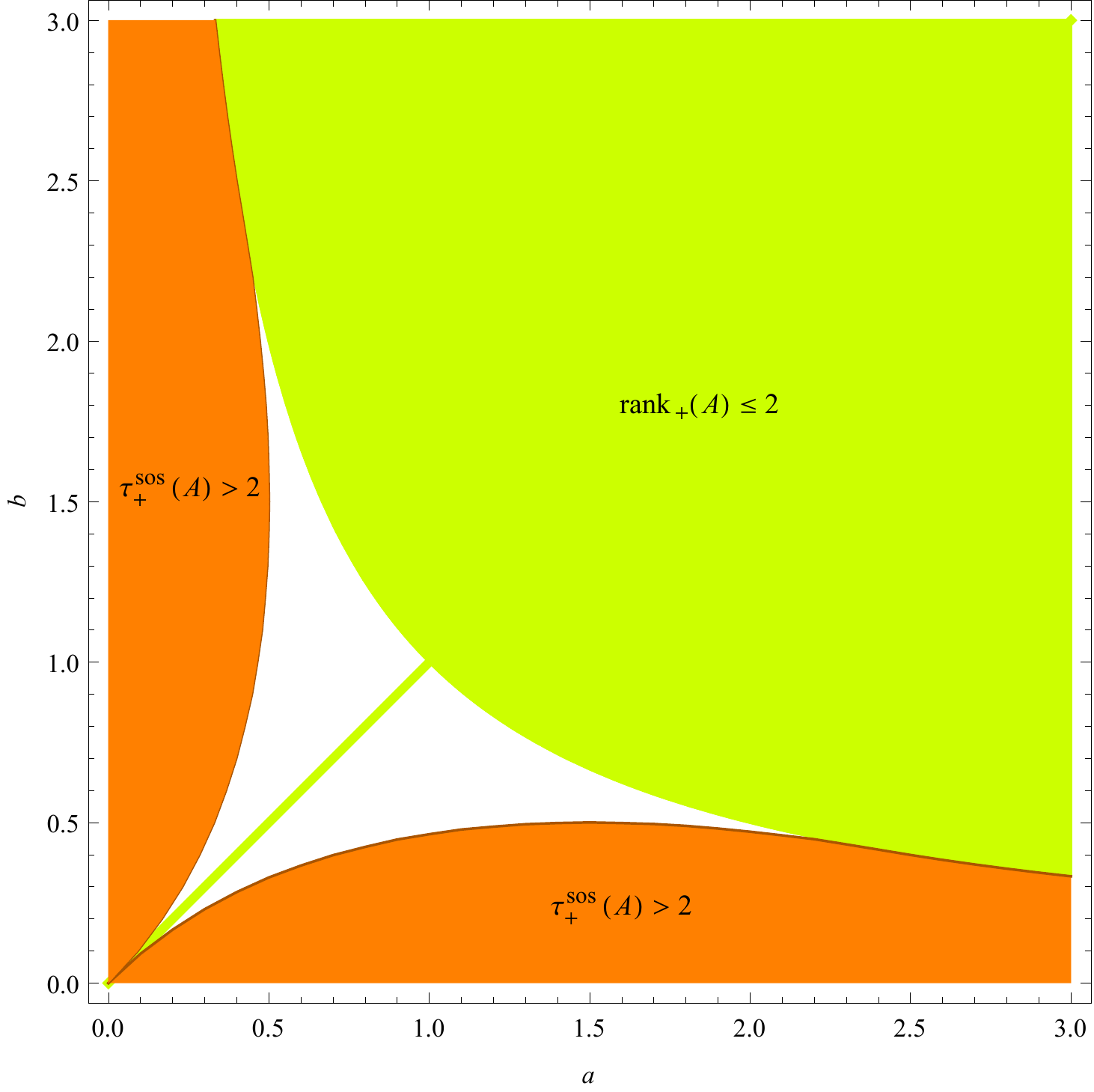}
\caption{Regions $(x,w)$ where $\tau_+^{\sos}(A) > 2$ and where $\rank_+(A) \leq 2$.}
\label{fig:region2x2x2tensor}
\end{figure}

\section{The cp-rank of completely positive matrices}
\label{sec:cprank}

In this section we apply the lower bounding technique to the cp-rank for completely positive matrices. For a reference on completely positive matrices we refer the reader to \cite{berman2003completely,dickinson2013copositive}.

\subsection{Definitions}

A symmetric matrix $A \in \S^n$ is called \emph{completely positive} if it admits a decomposition of the form:
\begin{equation}
 \label{eq:cp-factorization}
 A = \sum_{i=1}^r a_i a_i^T
\end{equation}
where each $a_i \in \RR^n$ is \emph{nonnegative}.
The cp-rank of $A$, denoted $\cprank(A)$ is the smallest $r$ for which $A$ admits a decomposition \eqref{eq:cp-factorization} where the number of rank-one terms is $r$. The cp-rank clearly satisfies
\[ \rank_+(A) \leq \cprank(A). \]
One can also show using Carath\'{e}odory theorem that $\cprank(A)$ is always bounded above by $n(n+1)/2$. More refined upper bounds on $\cprank$ exist, for example in terms of $\rank(A)$. We refer the reader to \cite[Section 3.2]{berman2003completely} for more information.

Observe that in any cp-factorization $A=\sum_{i=1}^k A_i$ where $A_i = a_i a_i^T$ with $a_i \geq 0$, we have
\begin{equation}
\label{eq:cpatomsineq}
 0 \leq A_i \leq A \quad \text{ and } \quad 0 \preceq A_i \preceq A
\end{equation}
where $\leq$ indicates componentwise inequality, and $\preceq$ indices inequality with respect to the positive semidefinite cone\footnote{Actually note that we can even write $0 \preceq_{\mathcal C} A_i \preceq_{\mathcal C} A$ where $\mathcal{C}$ is the cone of completely positive matrices. Since checking membership in the completely positive cone is hard \cite{dickinson2011computational}, we consider here only the tractable conditions \eqref{eq:cpatomsineq}.}. Following the ideas described in the previous sections, this leads to define the following set (where subscript `cp' indicates `completely positive'):
\[ \cA_{\cp}(A) = \left\{ R \in \S^n \; : \; \rank R \leq 1, \;\; 0\leq R \leq A, \;\; 0 \preceq R \preceq A\right\}. \]
If we introduce the quantity:
\[ 
\begin{aligned}
\tau_{\cp}(A) &= \; \min \quad \; t \; \quad \text{ subject to } \; A \in t\conv(\cA_{\cp}(A))\\
              &= \max_{L \text{ linear}} L(A) \; \text{ subject to } \; L(R) \leq 1 \; \forall R \in \cA_{\cp}(A)
\end{aligned}
\]
we can easily verify that:
\[ \tau_{\cp}(A) \leq \cprank(A). \]

\subsection{Semidefinite programming relaxation}
\label{sec:rank+sdp}

In this section we see how to obtain a semidefinite programming relaxation of $\tau_{\cp}(A)$. We proceed the same way as in section \ref{sec:tau+sdprelaxation} by constructing an over-relaxation of $\conv(\cA_{\cp}(A))$ which can be described using linear matrix inequalities.

Let $R \in \cA_{\cp}(A)$ and consider $\vec(R) \in \RR^{n^2}$ the vectorization of $R$ obtained by stacking the columns of $R$ on top of each other.
 We saw that the rank-one constraints on $R$ correspond to linear inequalities on the entries of $\vec(R)\vec(R)^T$, namely:
\begin{equation}
 \label{eq:cp-rankone-constraint}
 (\vec(R)\vec(R)^T)_{ij,kl} = (\vec(R)\vec(R)^T)_{il,kj}
\end{equation}
for all $(1,1) \leq (i,j) < (k,l) \leq (n,n)$. Furthermore, we saw that the componentwise inequality $0\leq R \leq A$ implies that:
\begin{equation}
 \label{eq:cp-componentineq-constraint}
 (\vec(R)\vec(R)^T)_{ij,ij} \leq R_{ij} A_{ij}
\end{equation}
for any $i,j\in [n]$.
Now, since we are dealing with cp-factorizations, we have the additional inequalities $0 \preceq R \preceq A$ which we can also exploit. An important observation here is that since $R$ is rank-one and positive semidefinite, we have:
\begin{equation}
 \label{eq:equalityrank1psd}
 \vec(R) \vec(R)^T = R \otimes R,
\end{equation}
where $\otimes$ denotes the Kronecker product of matrices. To see why \eqref{eq:equalityrank1psd} is true, let $x \in \RR^n$ such that $R = xx^T$. Let $\alpha \in [n^2]$ and $\beta \in [n^2]$ and let $(\alpha_1, \alpha_2) \in [n]^2$ and $(\beta_1,\beta_2) \in [n]^2$ be the unique pairs such that
\[ \alpha = (\alpha_2-1)n + \alpha_1 \quad \text{ and } \quad \beta = (\beta_2 - 1)n + \beta_1. \]
By definition of the operation $\vec$, we have $\vec(R)_{\alpha} = R_{\alpha_1,\alpha_2}$. Thus:
\[ (\vec(R) \vec(R)^T)_{\alpha,\beta} = R_{\alpha_1,\alpha_2} R_{\beta_1,\beta_2} = x_{\alpha_1} x_{\alpha_2} x_{\beta_1} x_{\beta_2}. \]
By definition of $R\otimes R$, we have:
\[ (R\otimes R)_{\alpha,\beta} = R_{\alpha_1,\beta_1} R_{\alpha_2,\beta_2} = x_{\alpha_1} x_{\alpha_2} x_{\beta_1} x_{\beta_2}. \]
This is true for any $\alpha,\beta \in [n^2]$ and thus we have $\vec(R)\vec(R)^T = R\otimes R$.
Now note that the matrix $R\otimes R$ satisfies $R\otimes R \preceq R\otimes A$: this is because $R\otimes A - R\otimes R = R\otimes (A-R) \succeq 0$ since $R \succeq 0$ and $A-R \succeq 0$ and the Kronecker product of positive semidefinite matrices is positive semidefinite. Thus we have the inequality:
\begin{equation}
 \label{eq:cp-psdineq-constraint}
 \vec(R) \vec(R)^T \preceq R \otimes A.
\end{equation}
If we now combine the observations above, we get the following over-relaxation of $\conv(\cA_{\cp}(A))$:
\[ \conv(\cA_{\cp}(A)) \subseteq \cA_{\cp}^{\sos}(A) \]
where
\begin{equation}
\label{eq:convAsos-cprank}
\begin{aligned}
\cA_{\cp}^{\sos}(A) = \Biggl\{ R \in \S^{n} \; : \; & \exists X \in \S^{n^2} \; \text{ such that } \; \begin{bmatrix}
1 & \vec(R)^T\\
\vec(R) & X
\end{bmatrix} \succeq 0\\
& \qquad \text{ and } X_{ij,ij} \leq R_{ij} A_{ij} \quad \forall i \in [m],j \in [n]\\
& \qquad \text{ and } X \preceq R \otimes A\\
& \qquad \text{ and } X_{ij,kl} - X_{il,kj} = 0 \quad \forall (1,1)\leq (i,j) < (k,l) \leq (m,n)\Biggr\}.
\end{aligned}
\end{equation}
This leads to the following relaxation $\tau_{\cp}^{\sos}(A)$ of $\tau_{\cp}(A)$:
\[ \tau_{\cp}^{\sos}(A) = \min \{ t \; : \; A \in t\cA^{\sos}_{\cp}(A)\}, \]
which satisfies:
\[ \tau_{\cp}^{\sos}(A) \leq \tau_{\cp}(A) \leq \cprank(A). \]
The function $\tau_{\cp}^{\sos}(A)$ can be computed by the following semidefinite program:
\begin{equation}
\begingroup
\renewcommand*{\arraystretch}{1.3}
\label{eq:taucpsos_min}
\begin{array}{rrl}
\tau_{\cp}^{\sos}(A) &= 
\;\;\text{min.} & t\\
&\text{s.t.} &
\begin{bmatrix} t & \vec(A)^T\\ \vec(A) & X \end{bmatrix} \succeq 0\\
&& X_{ij,ij} \leq A_{ij}^2 \quad \forall i,j \in [n]\\
&& X \preceq A \otimes A\\
&& X_{ij,kl} = X_{il,kj} \quad \forall (1,1) \leq (i,j) < (k,l) \leq (n,n)
\end{array}
\endgroup
\end{equation}
Note that this is very similar to the semidefinite program \eqref{eq:tau+sos_min} for $\tau_+^{\sos}(A)$ except for the additional constraint $X \preceq A\otimes A$.


\subsection{Properties}

The quantities $\tau_{\cp}$ and $\tau_{\cp}^{\sos}$ satisfy the same properties as those satisfied by $\tau_+$ and $\tau_{+}^{\sos}$ shown in section \ref{sec:tau+properties}. We summarize these properties below (the proofs are omitted since they are very similar to those from section \ref{sec:tau+properties}):

\begin{theorem}
Let $A$ be a completely positive matrix of size $n$.
\begin{enumerate}
\item Invariance under diagonal scaling: If $D$ is a diagonal matrix, with strictly positive entries on the diagonal, then $\tau_{\cp}(D A D) = \tau_{\cp}(A)$ and $\tau_{\cp}^{\sos}(D A D) = \tau_{\cp}^{\sos}(A)$.
\item Invariance under permutation: If $P$ is a permutation matrix, then $\tau_{\cp}(PAP^T) = \tau_{\cp}(A)$ and $\tau_{\cp}^{\sos}(P A P^T) = \tau_{\cp}^{\sos}(A)$.
\item Subadditivity: If $B$ is another completely positive matrix, then:
\[ \tau_{\cp}(A+B) \leq \tau_{\cp}(A) + \tau_{\cp}(B) \quad \text{ and } \quad \tau_{\cp}^{\sos}(A+B) \leq \tau_{\cp}^{\sos}(A) + \tau_{\cp}^{\sos}(B). \]
\item Monotonicity: If $B$ is a submatrix of $A$ (i.e., $B = A[I,I]$ for some $I\subseteq [n]$), then $\tau_{\cp}(B) \leq \tau_{\cp}(A)$ and $\tau_{\cp}^{\sos}(B) \leq \tau_{\cp}^{\sos}(A)$.
\item Block-diagonal composition: Let $B \in \S^{n'}$ be another completely positive matrix and define
\[ A \oplus B = \begin{bmatrix} A & 0\\ 0 & B \end{bmatrix} \in \S^{n+n'}. \]
Then
\[ \tau_{\cp}(A\oplus B) = \tau_{\cp}(A) + \tau_{\cp}(B) \quad \text{ and } \quad
   \tau_{\cp}^{\sos}(A\oplus B) = \tau_{\cp}^{\sos}(A) + \tau_{\cp}^{\sos}(B). \]
\end{enumerate}
\end{theorem}

\subsection{Comparison with existing lower bounds on cp-rank}

In this section we compare the lower bounds $\tau_{\cp}(A)$ and $\tau_{\cp}^{\sos}(A)$ to other existing lower bounds on cp-rank.

\subsubsection{The plain rank lower bound}

If $A$ is a completely positive matrix, an obvious lower bound to $\cprank(A)$ is $\rank(A)$. It turns out that $\tau_{\cp}^{\sos}(A)$ satisfies the remarkable property $\tau_{\cp}^{\sos}(A) \geq \rank(A)$:

\begin{theorem} Let $A$ be a completely positive matrix of size $n$. Then
\[ \tau_{\cp}^{\sos}(A) \geq \rank(A). \]
\end{theorem}
\begin{proof}
Let $(t,X)$ be the optimal solution of the semidefinite program \eqref{eq:taucpsos_min} where $t=\tau_{\cp}^{\sos}(A)$. Using Schur complement theorem we have that $\vec(A)\vec(A)^T \preceq tX$. Furthermore we also have $X \preceq A \otimes A$, and thus if we combine these two inequalities we get:
\[ \vec(A)\vec(A)^T \preceq t (A\otimes A). \]
Hence by Lemma \ref{lem:rankpsdvar} below we necessarily have $t \geq \rank(A)$.

\begin{lemma} 
\label{lem:rankpsdvar}
Let $A \in \S^n_+$ be a $n\times n$ positive semidefinite matrix. Then
\[ \rank(A) = \min \left\{ t \; : \; \vec(A) \vec(A)^T \preceq t A\otimes A \right\}. \]
\end{lemma}
\begin{proof}
Let $A = PDP^T$ be an eigenvalue decomposition of $A$ where $P$ is an orthogonal matrix and $D$ is a diagonal matrix where the diagonal elements are sorted in decreasing order. Let $r=\rank(A)=|\{i : D_{i,i} > 0\}|$ and denote by $I_r$ the identity matrix where only the first $r$ entries are set to 1 (the other entries are zero). For $t\geq 0$, the following equivalences hold:
\begin{equation}
\label{eq:gen-eigen-1}
\begin{aligned}
\vec(A)\vec(A)^T \preceq t A \otimes A \;\; & \Longleftrightarrow \;\; \vec(PDP^T)\vec(PDP^T)^T \preceq t (P\otimes P) (D\otimes D) (P^T \otimes P^T)\\
& \overset{\text{(a)}}{\Longleftrightarrow} \;\; (P\otimes P) \vec(D)\vec(D)^T (P^T \otimes P^T) \preceq t (P\otimes P) (D\otimes D) (P^T \otimes P^T)\\
& \overset{\text{(b)}}{\Longleftrightarrow} \;\; \vec(D)\vec(D)^T \preceq t (D\otimes D)\\
& \overset{\text{(c)}}{\Longleftrightarrow} \;\; \vec(I_r)\vec(I_r)^T \preceq t (I_r\otimes I_r)
\end{aligned}
\end{equation}
where in (a) we used the well-known identity $\vec(PDP^T) = (P\otimes P)\vec(D)$, in (b) we conjugated by $P\otimes P$ and in (c) we conjugated with $D^{-1/2} \otimes D^{-1/2}$ (where $(D^{-1/2})_{i,i} = (D_{ii})^{-1/2}$ if $D_{i,i} > 0$, else $(D^{-1/2})_{i,i} = 1$).
The lemma thus reduces to show that 
\[ \min \left\{ t \; : \; \vec(I_r)\vec(I_r)^T \preceq t (I_r \otimes I_r) \right\} = r. \]
 This is easy to see because $\vec(I_r)$ is an eigenvector of $\vec(I_r)\vec(I_r)^T$ with eigenvalue $r$, and it is the only eigenvector of $\vec(I_r)\vec(I_r)^T$ with a nonzero eigenvalue. Furthermore, $\vec(I_r)$ is also an eigenvector of $t(I_r \otimes I_r)$ with eigenvalue $t$. Thus the smallest $t$ such that $\vec(I_r)\vec(I_r)^T \preceq t (I_r \otimes I_r)$ is $r$.
\end{proof}
\if\mapr1\qed\fi
\end{proof}
\subsubsection{Combinatorial lower-bounds on cp-rank}

Given a $n\times n$ completely positive matrix $A$, let $G(A)$ be the graph whose adjacency matrix is $A$, i.e., $G$ has $n$ vertices and $i \in [n]$ and $j\in[n]$ are connected by an edge if $A_{i,j} > 0$. Observe that any cp-factorization of $A$:
\[ A = \sum_{i=1}^r a_i a_i^T, \]
where $a_i \geq 0$ yields a covering of the edges of $G(A)$ using $r$ cliques of $G(A)$. Indeed, the support of each rank-one term $a_i a_i^T$ corresponds to a clique of $G(A)$, and each nonzero entry of $A$ (i.e., each edge of $G(A)$) is covered by at least one such clique. This simple observation yields the following lower-bound on the cp-rank of $A$:
\[ \cprank(A) \geq c(G(A)), \]
where $c(G(A))$ is the \emph{edge clique-cover number} of $G(A)$, i.e., the smallest number of cliques of $G(A)$ needed to cover the edges of $G(A)$. Note that $c(G)$ is NP-hard to compute in general. Also note that for a graph $G=(V,E)$, $c(G)$ is the solution of the following integer program where the variables $x_C$ are indexed by the cliques $C$ of $G$:
\[
c(G) = 
\begin{array}[t]{ll}
\text{min} & \sum_{C} x_C\\
\text{s.t.} & x_C \in \{0,1\} \quad \forall C \text{ clique of $G$}\\
                  & \sum_{C \; : \; e \in C } x_C \geq 1 \quad \forall e \in E.
\end{array}
\]
Define the \emph{fractional edge-clique cover number} of $G$, denoted $c_{\fra}(G)$ to be the linear programming relaxation of the integer program above, where the integer constraints $x_C \in \{0,1\}$ are replaced by $x_C \geq 0$ \cite{scheinerman1999fractional}:
\begin{equation}
\label{eq:cfraclp}
c_{\fra}(G) = 
\begin{array}[t]{ll}
\text{min} & \sum_{C} x_C\\
\text{s.t.} & x \geq 0\\
                  & \sum_{C \; : \; e \in C } x_C \geq 1 \quad \forall e \in E
\end{array}
\end{equation}
%
Note that the linear program \eqref{eq:cfraclp} is hard to compute in general since it has an exponential number of variables. Clearly we have $c_{\fra}(G(A)) \leq c(G(A))$. Actually, Lov\'{a}sz showed in \cite{lovasz1975ratio} that $c_{\fra}(G)$ is always within a $\ln |E|$ factor from $c(G)$ for any graph $G=(V,E)$:
\[ \frac{1}{1+\ln|E|} c(G) \leq c_{\fra}(G) \leq c(G). \]

We will now rewrite the linear program \eqref{eq:cfraclp} for $c_{\fra}(G)$ in a slightly different way in order to show its connection with the quantity $\tau_{\cp}(A)$ (we will in fact show in a theorem below that $\tau_{\cp}(A) \geq c_{\fra}(A)$). Define $\cA_{\ecc}(A)$ to be the set of adjacency matrices representing cliques in $G(A)$:
\[ \cA_{\ecc}(A) = \left\{ R = bb^T \text{ where } b \in \{0,1\}^n \text{ and $R$ is monochromatic for $A$ } \right\}. \]
We can rewrite the fractional edge-clique cover number of $G(A)$ as follows:
\begin{equation}
\label{eq:fracecc}
 c_{\fra}(G(A)) = 
\begin{array}[t]{ll}
\text{min} & t\\
\text{s.t.} & \exists Y \in t \conv(\cA_{\ecc}(A)) \; \text{ s.t. } \; \forall (i,j), \;\; A_{i,j} > 0 \; \Rightarrow \; Y_{i,j} \geq 1
\end{array}
\end{equation}
We prove the following:
\begin{theorem}
If $A$ is a completely positive matrix, then:
\[ \tau_{\cp}(A) \geq c_{\fra}(G_A). \]
\end{theorem}
\begin{proof}
Let $t = \tau_{\cp}(A)$ and $X \in \conv(\cA_{\cp}(A))$ such that $A = t X$. Consider the decomposition of $X$:
\[ X = \sum_{k=1}^r \lambda_k X_k, \]
where $X_k \in \cA_{\cp}(A)$, $\lambda_k \geq 0$ and $\sum_{k=1}^r \lambda_k = 1$. Let $R_k = \supp(X_k)$ (i.e., $R_k$ is obtained by replacing the nonzero entries of $X_k$ with ones) and observe that $R_k \in \cA_{\ecc}(A)$. Define
\[ Y = t \sum_{k=1}^r \lambda_k R_k \in t \conv(\cA_{\ecc}(A)) \]
Observe that for any $(i,j)$ such that $A_{i,j} > 0$ we have:
\[ Y_{i,j} = t \sum_{k:X_k[i,j]>0} \lambda_k \underbrace{R_k[i,j]}_{=1} \overset{(a)}{\geq} t \sum_{k:X_k[i,j]} \lambda_k \frac{X_k[i,j]}{A_{i,j}} \overset{(b)}{=} \frac{A_{i,j}}{A_{i,j}} = 1 \]
where in (a) we used the fact that $X_k \leq A$ (by definition of $X_k \in \cA_{\cp}(A)$) and in (b) we used the fact that $A = t\sum_{k} \lambda_k X_k$. Thus this shows that $(t,Y)$ is feasible for the optimization program defining $c_{\fra}(G(A))$ and thus we have $c_{\fra}(G(A)) \leq t = \tau_{\cp}(A)$.
\end{proof}

\subsection{Example}

Consider the following matrix parameterized by $a,b \geq 0$:
\begin{equation}
\label{eq:example_matrix_A_cp}
A = \begin{bmatrix}
3+a & 0 & 1 & 1 & 1\\
0 & 3+a & 1 & 1 & 1\\
1 & 1 & 2+b & 0 & 0\\
1 & 1 & 0 & 2+b & 0\\
1 & 1 & 0 & 0 & 2+b
\end{bmatrix}.
\end{equation}
When $a,b \geq 0$, the matrix $A$ is nonnegative and diagonally dominant and hence is completely positive \cite[Theorem 2.5]{berman2003completely}. One can show that the cp-rank of $A$ is equal to 6 for any $a,b \geq 0$. Indeed observe that $A$ is the adjacency matrix of the graph $K_{2,3}$ (complete bipartite graph) which has edge-clique cover number of 6, and thus necessarily $\cprank(A) \geq 6$. Also it is known\footnote{It is conjectured that any completely positive matrix of size $n\times n$ has cprank $\leq n^2/4$. The conjecture is known to be true for $n=5$, and this means that any $5\times 5$ completely positive matrix has cprank $\leq 6$. The conjecture is known as the DJL conjecture \cite[p.157]{berman2003completely}} that any $5\times 5$ completely positive matrix has cp-rank $\leq$ 6 \cite[Theorem 3.12]{berman2003completely}.

We have computed the value of $\tau_{\cp}^{\sos}(A)$ for different values of $a,b$ and we show the result of these computations in Figure \ref{fig:example_cprank_plots}: the left figure shows the plot of $\tau_{\cp}^{\sos}(A)$ as a function of $a,b$; the right figure shows the region of values $a,b$ where $\tau_{\cp}^{\sos}(A) > 5$.

\begin{figure}[!ht]
  \centering
  \includegraphics[width=8cm]{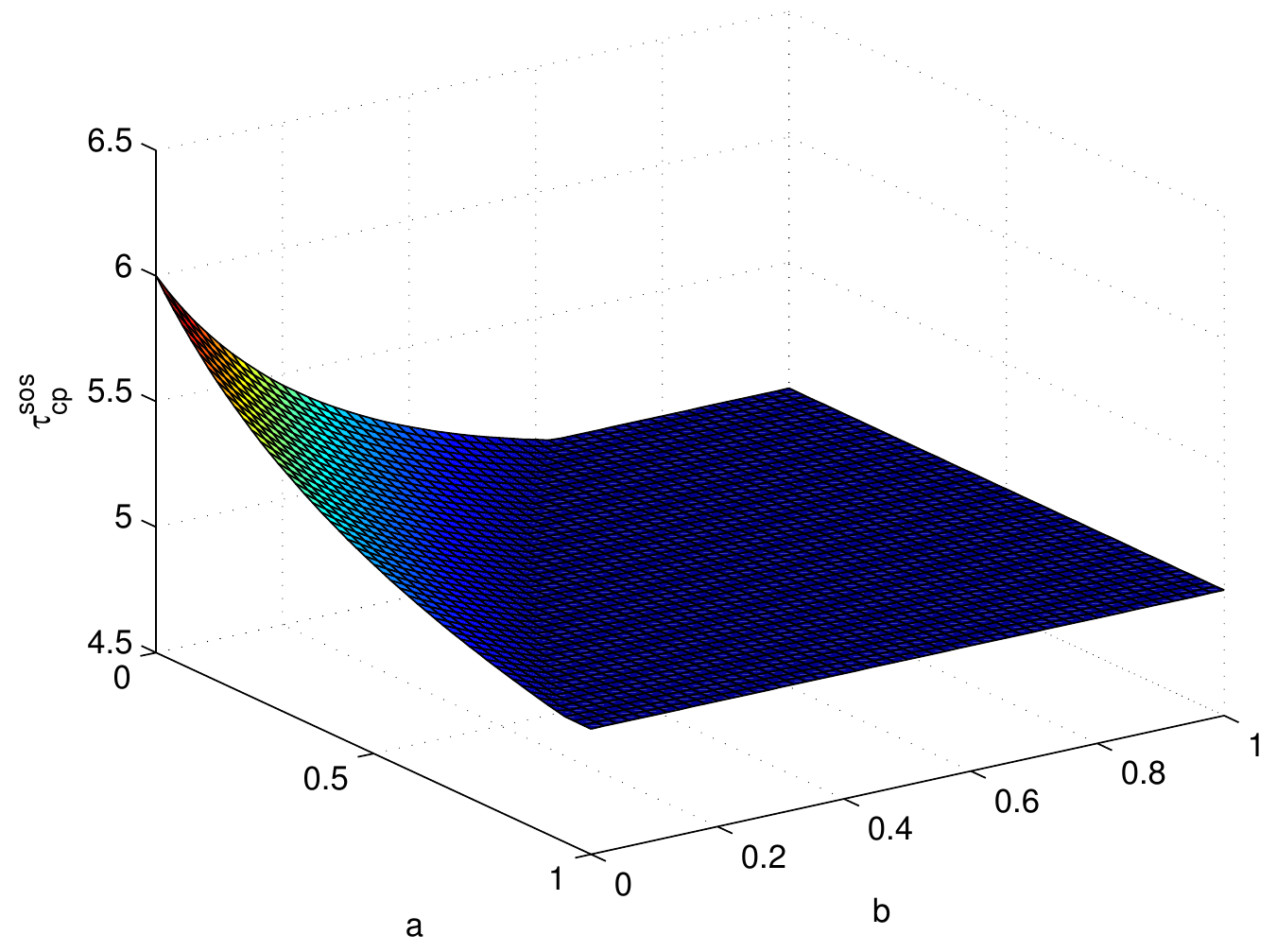}\quad
  \includegraphics[width=8cm]{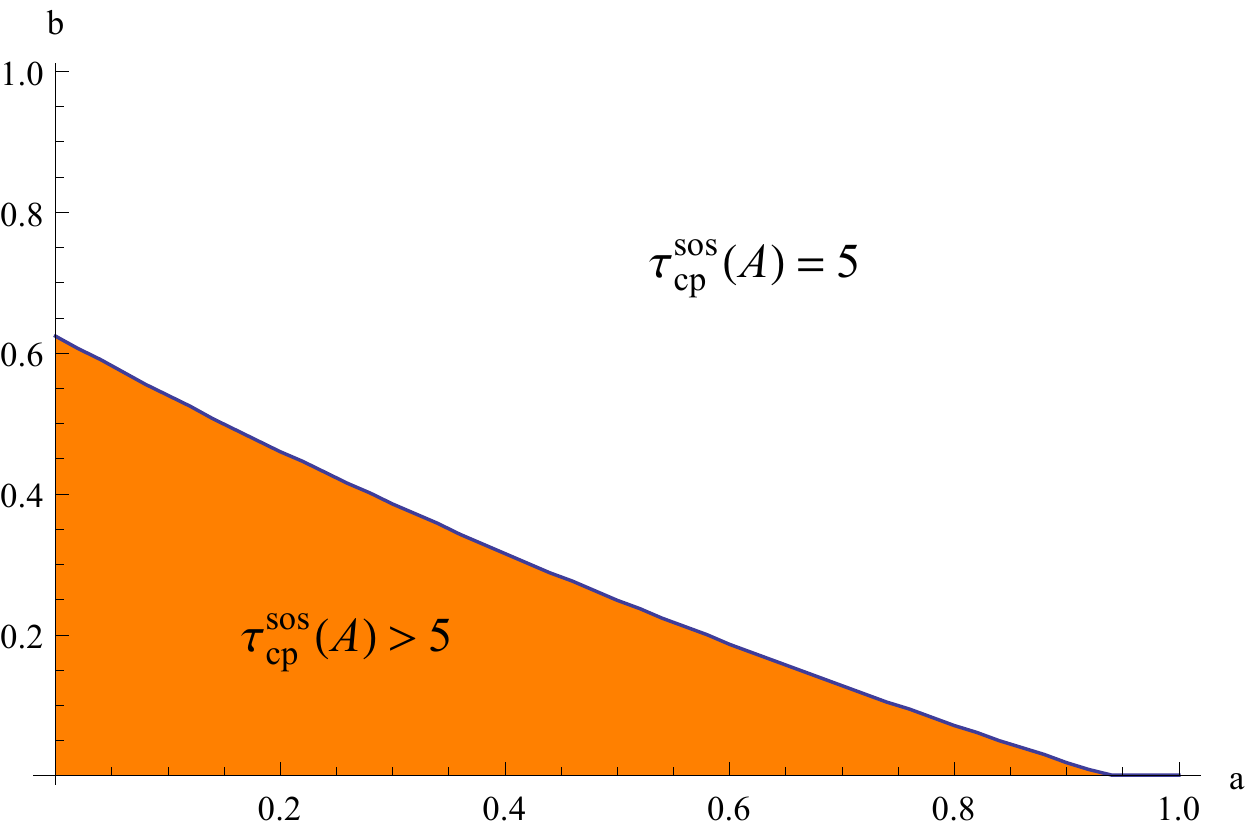}
  \caption{Left: Plot of $\tau_{\cp}^{\sos}(A)$ as a function of $a,b$ (the matrix $A$ is defined in \eqref{eq:example_matrix_A_cp}). Right: Region where $\tau_{\cp}^{\sos}(A) > 5$. Note that for $a=b=0$ we have $\tau_{\cp}^{\sos}(A) = 6$, and for $(a,b)$ outside the colored region, we have $\tau_{\cp}^{\sos}(A) = 5$}
  \label{fig:example_cprank_plots}
\end{figure}

\section{Summary and conclusion}

In this paper we proposed a general method based on convex optimization to obtain lower bounds on so-called \emph{atomic cone ranks}.
We focused on two important special cases which are the \emph{nonnegative rank} and \emph{cp-rank}. In these cases we saw that our lower bound improves on the existing bounds and enjoy in addition appealing structural properties. There are also other examples of \emph{atomic cone ranks} that one could study using the framework developed in this paper. For example, one interesting application mentioned earlier is to obtain lower bounds on the sizes of cubature formulas \cite{konig1999cubature}.

Note that there are other notions of \emph{cone ranks} which do not fit in the atomic framework described here. One example which has received a lot of attention recently is the \emph{psd rank} defined in \cite{gouveia2011lifts} and which has applications in semidefinite lifts of polytopes. Given a nonnegative matrix $A \in \RR^{m\times n}_+$, the psd rank of $A$ is the smallest $r$ for which we can find $r\times r$ positive semidefinite matrices $U_i, V_j$ such that $A_{ij} = \langle U_i, V_j \rangle$ for all $i,j$. Unlike the nonnegative rank, the psd rank is not an \emph{atomic rank} since the matrices with psd rank one are precisely the matrices that have nonnegative rank one. This non-atomic feature makes the psd rank more difficult to study, and there are currently no good methods known to obtain strong lower bounds on the psd rank.


\appendix

\section{Proof of properties of $\tau_+$ and $\tau_+^{\sos}$}
\label{sec:appendix-props}

\subsection{Invariance under permutation}

 The proof of invariance under permutation is very similar to the one for invariance under diagonal scaling. To prove the claim for $\tau_+$ one proceeds by showing that the set of atoms $\cA_+(A')$ of $A'=P_1 A P_2$ can be obtained from the atoms of $A$ by applying the permutations $P_1$ and $P_2$, namely:
\[ \cA_+(A') = \{ P_1 R P_2 \; : \; R \in \cA_+(A) \} =: P_1 \cA_+(A) P_2. \]
For the SDP relaxation we also use the same idea as the previous proof by constructing a certificate $L'$ for $A'$ using the certificate $L$ for $A$. We omit the details here since they are very similar to the previous proof.


\subsection{Subadditivity}

\begin{enumerate}
\item We first prove the subadditivity property for $\tau_+$, i.e., $\tau_+(A+B) \leq \tau_+(A)+\tau_+(B)$. Observe that we have 
\begin{equation}
\label{eq:inclusion}
\cA_+(A) \cup \cA_+(B) \subseteq \cA_+(A+B).
\end{equation}
Indeed if $R \in \cA_+(A)$, i.e., $R$ is rank-one and $0\leq R \leq A$, then we also have $0 \leq R \leq A+B$ (since $B$ is nonnegative) and thus $R \in \cA_+(A+B)$. Thus this shows $\cA_+(A) \subseteq \cA_+(A+B)$, and the same reason gives $\cA_+(B) \subseteq \cA_+(A+B)$, and thus we get \eqref{eq:inclusion}. By definition of $\tau_+(A)$ and $\tau_+(B)$, we know there exist decompositions of $A$ and $B$:
\[ A = \sum_{i} \alpha_i R_i \]
and
\[ B = \sum_{j} \beta_j S_j \]
where $R_i \in \cA_+(A)$ for all $i$ and $\sum_{i} \alpha_i = \tau_+(A)$, and $S_j \in \cA_+(B)$ for all $j$ and $\sum_{j} \beta_j = \tau_+(B)$. Thus this leads to:
\[ A+B = \sum_{i} \alpha_i R_i + \sum_{j} \beta_j S_j \]
where $R_i \in \cA_+(A+B)$ and $S_j \in \cA_+(A+B)$ for all $i$ and $j$. This decomposition shows that
\[ \tau_+(A+B) \leq \tau_+(A) + \tau_+(B). \]
\item We now prove the property for $\tau_+^{\sos}$. Let $(t,X)$ and $(t',X')$ be the optimal points of the semidefinite program \eqref{eq:tau+sos_min} for $A$ and $B$ respectively (i.e., $t = \tau_+^{\sos}(A)$ and $t' = \tau_+^{\sos}(B)$). It is not hard to see that $(t+t',X+X')$ is feasible for the semidefinite program that defines $\tau_+^{\sos}(A+B)$ (in particular we use the fact that since $A$ and $B$ are nonnegative we have $A_{ij}^2 + B_{ij}^2 \leq (A_{ij}+B_{ij})^2$). Thus this shows that $\tau_+^{\sos}(A+B) \leq \tau_+^{\sos}(A) + \tau_+^{\sos}(B)$.
\end{enumerate}

\subsection{Product}
\label{app:product}

\begin{enumerate}
\item We first show the property for $\tau_+$. We need to show that $\tau_+(AB) \leq \min(\tau_+(A),\tau_+(B))$. To see why note that if $R \in \cA_+(A)$, then $RB \in \cA_+(AB)$. thus if we have $A = \sum_{i} \alpha_i R_i$ with $R_i \in \cA_+(A)$ and $\sum_i \alpha_i = \tau_+(A)$, then we get $AB = \sum_i \alpha_i R_i B$ where each $R_iB \in \cA_+(AB)$ and thus $\tau_+(AB) \leq \sum_i \alpha_i = \tau_+(A)$. The same reasoning shows that $\tau_+(AB) \leq \tau_+(B)$, and thus we get $\tau_+(AB) \leq \min(\tau_+(A),\tau_+(B))$.
\item We now prove the property for $\tau_+^{\sos}$, i.e., we show $\tau_+^{\sos}(AB) \leq \min(\tau_+^{\sos}(A),\tau_+^{\sos}(B))$. We will show here that $\tau_+^{\sos}(AB) \leq \tau_+^{\sos}(A)$, and a similar reasoning can then be used to show $\tau_+^{\sos}(AB) \leq \tau_+^{\sos}(B)$.

Let $(t,X)$ be the optimal point of the semidefinite program \eqref{eq:tau+sos_min} that defines $\tau_+^{\sos}(A)$, i.e., $t = \tau_+^{\sos}(A)$.
We will show that the pair $(t,\tilde{X})$ with
\[ \tilde{X} = (B^T \otimes I_m) X (B \otimes I_m), \]
is feasible for the semidefinite program that defines $\tau_+^{\sos}(AB)$ and thus this will show that $\tau_+^{\sos}(AB) \leq t = \tau_+^{\sos}(A)$.

Observe that we have $\vec(AB) = (B^T \otimes I_m) \vec(A)$ thus:
\[
\begin{bmatrix}
t & \vec(AB)^T\\
\vec(AB) & \tilde{X}
\end{bmatrix}
=
\begin{bmatrix}
1 & 0\\
0 & B^T \otimes I
\end{bmatrix}
\begin{bmatrix}
t & \vec(A)^T\\
\vec(A) & X
\end{bmatrix}
\begin{bmatrix}
1 & 0\\
0 & B \otimes I
\end{bmatrix}
\]
and thus this shows that the matrix
\[ 
\begin{bmatrix}
t & \vec(AB)^T\\
\vec(AB) & \tilde{X}
\end{bmatrix}
\]
is positive semidefinite.

Using the definition of Kronecker product one can verify that the entries of $\tilde{X}$ are given by:
\[ \tilde{X}_{ij,kl} = \sum_{\alpha,\beta=1}^{m'} B_{\alpha j} B_{\beta l} X_{i \alpha, k \beta}. \]
Using this formula we easily verify that $\tilde{X}$ satisfies the rank-one equality constraints:
\[ \tilde{X}_{ij,kl} = \tilde{X}_{il,kj} \]
since $X$ itself satisfies the constraints.

Finally it remains to show that $\tilde{X}_{ij,ij} \leq (AB)_{ij}^2$. For this we need the following simple lemma:
\begin{lemma}
\label{lem:offdiag-ineqs}
Let $(t,X)$ be a feasible point for the semidefinite program \eqref{eq:tau+sos_min}. Then $X_{ij,kl} \leq A_{ij}A_{kl}$ for any $i,j,k,l$.
\end{lemma}
\begin{proof}
Consider the $2\times 2$ principal submatrix of $X$:
\[ \begin{bmatrix} X_{ij,ij} & X_{ij,kl}\\ X_{kl,ij} & X_{kl,kl} \end{bmatrix}. \]
We know that $X_{ij,ij} \leq A_{ij}^2$ and $X_{kl,kl} \leq A_{kl}^2$. Furthermore since $X$ is positive semidefinite we have $X_{ij,ij} X_{kl,kl} - X_{ij,kl}^2 \geq 0$. Thus we get that:
\[ X_{ij,kl}^2 \leq X_{ij,ij} X_{kl,kl} \leq (A_{ij} A_{kl})^2. \]
Thus since $A_{ij} A_{kl} \geq 0$ we have $X_{ij,kl} \leq A_{ij} A_{kl}$.
\if\mapr1\qed\fi
\end{proof}
Using this lemma we get:
\[ \tilde{X}_{ij,ij} = \sum_{\alpha,\beta=1}^{m'} B_{\alpha j} B_{\beta j} X_{i \alpha, i \beta} \leq \sum_{\alpha,\beta=1}^{m'} B_{\alpha j} B_{\beta j} A_{i \alpha} A_{i \beta} = ((AB)_{ij})^2 \]
which is what we want.
\end{enumerate}

\subsection{Monotonicity}

Here we prove the monotonicity property of $\tau_+$ and $\tau_+^{\sos}$. More precisely we show that if $A \in \RR^{m\times n}_+$ is a nonnegative matrix, and $B$ is a submatrix of $A$, then $\tau_+(B) \leq \tau_+(A)$ and $\tau_+^{\sos}(B) \leq \tau_+^{\sos}(A)$.
\begin{enumerate}
\item We prove the claim first for $\tau_+$. Let $I\subseteq [m]$ and $J\subseteq [n]$ such that $B=A[I,J]$ (i.e., $B$ is obtained from $A$ by keeping only the rows in $I$ and the columns in $J$). Let $X \in \conv \cA_+(A)$ such that $A = \tau_+(A) X$. Define $Y = X[I,J]$ and note that $Y \in \conv(\cA_+(B))$. Furthermore observe that we have $B = A[I,J] = \tau_+(A) X[I,J] = \tau_+(A) Y$. Hence, since $Y \in \conv \cA_+(B)$, this shows, by definition of $\tau_+(B)$ that $\tau_+(B) \leq \tau_+(A)$.
\item We prove the claim now for the semidefinite programming relaxation $\tau_+^{\sos}$. As above, let $I\subseteq [m]$ and $J\subseteq [n]$ such that $B=A[I,J]$. Let $(t,X)$ be the optimal point in \eqref{eq:tau+sos_min} for the matrix $A$. It is easy to see that $(t,X[I,J])$ is feasible  for the semidefinite program \eqref{eq:tau+sos_min} for the matrix $B=A[I,J]$. Thus this shows that $\tau_+^{\sos}(B) \leq \tau_+^{\sos}(A)$.
\end{enumerate}

\subsection{Block-diagonal matrices}
In this section we prove that if $A \in \RR^{m\times n}_+$ and $B \in \RR^{m'\times n'}_+$ are two nonnegative matrices and $A\oplus B$ is the block-diagonal matrix:
\[ A \oplus B = \begin{bmatrix} A & 0\\ 0 & B \end{bmatrix}, \]
then
\[ \tau_+(A\oplus B) = \tau_+(A) + \tau_+(B) \quad \text{ and } \quad
   \tau_+^{\sos}(A\oplus B) = \tau_+^{\sos}(A) + \tau_+^{\sos}(B) \]

\begin{enumerate}
\item We first prove the claim for the quantity $\tau_+$. Observe that the set $\cA_+(A\oplus B)$ is equal to:
\begin{equation}
 \label{eq:A+AoplusB}
 \cA_+(A\oplus B) = \left\{ \begin{bmatrix} R & 0\\ 0 & 0 \end{bmatrix} \; : \; R \in \cA_+(A) \right\} \cup \left\{ \begin{bmatrix} 0 & 0\\ 0 & R' \end{bmatrix} \; : \; R' \in \cA_+(B) \right\}.
\end{equation}
Indeed any element in $\cA_+(A\oplus B)$ must have the off-diagonal blocks equal to zero (since the off-diagonal blocks of $A\oplus B$ are zero), and thus by the rank-one constraint at least one of the diagonal blocks is also equal to zero. Thus this shows that $\cA_+(A\oplus B)$ decomposes as in \eqref{eq:A+AoplusB}.

We start by showing $\tau_+(A \oplus B) \geq \tau_+(A) + \tau_+(B)$. Let $Y \in \conv \cA_+(A \oplus B)$ such that
\[ A\oplus B = \tau_+(A\oplus B) Y. \]
Since $\cA_+(A\oplus B)$ has the form \eqref{eq:A+AoplusB}, we know that $Y$ can be decomposed as:
\[ Y = \sum_{i=1}^r \lambda_i \begin{bmatrix} R_i & 0\\ 0 & 0 \end{bmatrix} + \sum_{j=1}^{r'} \mu_{j} \begin{bmatrix} 0 & 0\\ 0 & R'_{i'} \end{bmatrix}, \]
where $R_i \in \cA_+(A), R'_{j} \in \cA_+(B)$ and $\sum_i \lambda_i + \sum_{j} \mu_j = 1$ with $\lambda, \mu \geq 0$. Note that since $A\oplus B = \tau_+(A\oplus B) Y$ we have:
\[ A = \tau_+(A\oplus B) \sum_{i=1}^r \lambda_i R_i, \]
and
\[ B = \tau_+(A\oplus B) \sum_{j=1}^{r'} \mu_j R'_j. \]
Hence $\tau_+(A) \leq \tau_+(A\oplus B) \sum_{i=1}^r \lambda_i$ and $\tau_+(B) \leq \tau_+(A\oplus B) \sum_{j=1}^{r'} \mu_j$ and we thus get:
\[ \tau_+(A) + \tau_+(B) \leq \tau_+(A\oplus B) \left(\sum_{i=1}^r \lambda_i + \sum_{j=1}^{r'} \mu_j\right) = \tau_+(A\oplus B). \]

We now prove the converse inequality, i.e., $\tau_+(A\oplus B) \leq \tau_+(A) + \tau_+(B)$: Let $t=\tau_+(A)$, $t' = \tau_+(B)$ and $X \in \conv \cA_+(A), X' \in \conv \cA_+(B)$ such that $A = t X$ and $B = t' X'$. Define the matrix 
\[ Y = \begin{bmatrix} \frac{t}{t+t'} X & 0\\ 0 & \frac{t'}{t+t'} X' \end{bmatrix}, \]
and note that $A \oplus B = (t+t') Y$. If we show that $Y \in \conv \cA_+(A\oplus B)$ then this will show that $\tau_+(A\oplus B) \leq t + t'$. We can rewrite $Y$ as:
\[ Y = \frac{t}{t+t'} \begin{bmatrix} X & 0\\ 0 & 0 \end{bmatrix} + \frac{t'}{t+t'} \begin{bmatrix} 0 & 0\\ 0 & X'\end{bmatrix}, \]
and it is easy to see from this expression that $Y \in \conv \cA_+(A\oplus B)$.

We have thus proved that $\tau_+(A\oplus B) = \tau_+(A) + \tau_+(B)$.

\item We now prove the claim for the SDP relaxation $\tau_+^{\sos}$. Let $a = \vec(A)$ and $b=\vec(B)$. Since the matrix $A \oplus B$ has zeros on the off-diagonal, the SDP defining $\tau_+^{\sos}(A\oplus B)$ can be simplified and we can eliminate the zero entries from the program. One can show that after the simplification we get that $\tau_+^{\sos}(A\oplus B)$ is equal to the value of the SDP below:
\begin{equation}
\label{eq:tausos-blkdiag}
\begingroup
\renewcommand*{\arraystretch}{1.3}
\begin{array}[t]{ll}
\text{minimize} & t\\
\text{subject to} & \begin{bmatrix} t & a^T & b^T\\ a & X & 0\\ b & 0 & X' \end{bmatrix} \succeq 0\\
& X_{ij,ij} \leq A_{ij}^2 \quad \forall (i,j) \in [m]\times [n]\\
& X_{ij,kl} = X_{il,kj}  \quad 1\leq i < k \leq m \text{ and } 1\leq j < l \leq n\\
& X'_{i'j',i'j'} \leq B_{i'j'}^2 \quad \forall (i',j') \in [m']\times [n']\\
& X'_{i'j',k'l'} = X_{i'l',k'j'}  \quad 1\leq i' < k' \leq m' \text{ and } 1\leq j' < l' \leq n'
\end{array}
\endgroup
\end{equation}
It is well-known (see e.g., \cite{grone1984positive}) that the following equivalence always holds:
\[
\begin{bmatrix} t & a^T & b^T\\ a & X & 0\\ b & 0 & X' \end{bmatrix} \succeq 0
\;\;
\Longleftrightarrow
\;\;
\exists t_1,t_2 \; : \; t_1+t_2 = t, \quad \begin{bmatrix} t_1 & a^T\\ a & X \end{bmatrix} \succeq 0, \quad \begin{bmatrix} t_2 & b^T\\ b & X' \end{bmatrix} \succeq 0
\]
Using this equivalence, the semidefinite program \eqref{eq:tausos-blkdiag} becomes:
\begin{equation}
\label{eq:tausos-blkdiag-2}
\begin{array}[t]{ll}
\text{minimize} & t_1 + t_2\\
\text{subject to} & \begin{bmatrix} t_1 & a^T\\ a & X \end{bmatrix} \succeq 0\\
& X_{ij,ij} \leq A_{ij}^2 \quad \forall (i,j) \in [m]\times [n]\\
& X_{ij,kl} = X_{il,kj}  \quad 1\leq i < k \leq m \text{ and } 1\leq j < l \leq n\\
& \begin{bmatrix} t_2 & b^T\\ b & X' \end{bmatrix} \succeq 0\\
& X'_{i'j',i'j'} \leq B_{i'j'}^2 \quad \forall (i',j') \in [m']\times [n']\\
& X'_{i'j',k'l'} = X_{i'l',k'j'}  \quad 1\leq i' < k' \leq m' \text{ and } 1\leq j' < l' \leq n'
\end{array}
\end{equation}
The semidefinite program is decoupled and it is easy to see that its value is equal to $\tau_{+}^{\sos}(A) + \tau_{+}^{\sos}(B)$.
\end{enumerate}

\if\mapr1
\bibliographystyle{spmpsci}
\else
\bibliographystyle{alpha}
\fi
\bibliography{../../../bib/nonnegative_rank}

\end{document}